\documentclass[11pt,a4paper,twoside]{article}
\usepackage{latexsym,amsfonts,amsmath, amsthm, amssymb}
\usepackage{fullpage}
\usepackage{xcolor,bm, bbm}
\usepackage[utf8x]{inputenc}
\usepackage{array}
\usepackage{color}
\usepackage{mathrsfs} 
\usepackage{comment}
\usepackage{mathtools}
\usepackage{fourier}
\usepackage{hyperref} 
\usepackage{graphicx}

\definecolor{col1}{rgb}{0.3.0.4,0.6}

\renewcommand{\P}{\mathbb P} 

\newcommand{\E}{\mathbb{E}}
\newcommand{\V}{\mathbb{V}}
 
\newcommand{\Q}{\mathbb{Q}}

\newcommand{\R}{\mathbb R}
\newcommand{\N}{\mathbb N}

\newcommand{\Var}{\mathrm{Var}}

\newcommand{\vol}{\mathrm{vol}}

\def\dint{\textup{d}}

\newcommand{\B}{\ensuremath{{\mathbb B}}}

\newcommand{\prosec}{\diamondsuit}
\newcommand{\proj}{\big|}

\newcommand{\Cov}{\operatorname{Cov}}

\newcommand{\Sm}{{\mathbb{S}^{m-1}}}

\newtheorem{thm}{Theorem}[section]

\newtheorem{lemma}[thm]{Lemma}
\newtheorem{df}[thm]{Definition}
\newtheorem{proposition}[thm]{Proposition}

\newtheorem{thmalpha}{Theorem}

{
	\theoremstyle{definition}
	
	\newtheorem{rmk}[thm]{Remark}
	
}

\allowdisplaybreaks
\begin{document}
	
	\title{\bf Limit Theorems for the Volume of\\ Random Projections and Sections of $\ell_p^N$-balls}
	
	\medskip
	
	\author{Joscha Prochno, Christoph Thäle, Philipp Tuchel}
	
	
	
	\date{}
	
	\maketitle
	
	\begin{abstract}
		\small Let $\mathbb{B}_p^N$ be the $N$-dimensional unit ball corresponding to the $\ell_p$-norm. For each $N\in\mathbb N$ we sample a uniform random subspace $E_N$ of  fixed dimension $m\in\mathbb{N}$ and consider the volume of $\mathbb{B}_p^N$ projected onto $E_N$ or intersected with $E_N$. We also consider geometric quantities other than the volume such as the intrinsic volumes or the dual volumes. In this setting we prove central limit theorems, moderate deviation principles, and large deviation principles as $N\to\infty$. Our results provide a complete asymptotic picture. In particular, they generalize and complement a result of Paouris, Pivovarov, and Zinn [A central limit theorem for projections of the cube, Probab. Theory Related Fields. 159 (2014), 701-719] and another result of Adamczak, Pivovarov, and Simanjuntak [Limit theorems for the volumes of small codimensional random sections of $\ell_p^n$-balls, Ann. Probab. 52 (2024), 93-126].

		\medspace
	\end{abstract}
	
	\tableofcontents
	
	\section{Introduction} \label{sec:intro}
	
	The asymptotic theory of convex bodies plays a major role in the fields of Asymptotic Geometric Analysis and High-Dimensional Probability, two rather young branches of mathematics having their origin in the Local Theory of Banach spaces and Probability in Banach Spaces; we refer to \cite{AGAI2015} and \cite{vershynin2018high}. The two fields and, in particular, the methods used and developed during the last decades are closely linked to convex and discrete geometry, functional analysis, and probability theory, and have a variety of applications in, for instance, compressed sensing and image processing \cite{CGLP2012,FR2013}, information-based complexity \cite{HKNPU2021,HPU2019}, and machine learning \cite{ES2016,MP2004,MS2003}, to mention just a few. 
	
	The probabilistic perspective on high-dimensional convex bodies and structures is an extremely powerful one and has led to several major breakthrough results. We would like to mention here the central limit theorem for convex bodies obtained by Klartag \cite{K2007} or Eldan's stochastic localization method \cite{E2013}, which eventually led to the spectacular solution of the slicing conjecture in \cite{klartaglehec2024} (see also \cite{Chen2021,guan2024,KL2022,LV2024} for previous works in this direction).

	In the last decades a number of weak limit theorems related to convex bodies and high-dimensional random geometric structures and quantities have appeared and we refer the reader to the following non-exhaustive list and the references cited therein \cite{AGPT2019,barany2007central,HLSch2016,HR2005,JP2022,kablu19high,Rei2005,SchSch1991,Schmuck2001,ST2023_schatten,TW2018}. Less than a decade ago, the central limit perspective has been complemented by moderate and large deviation principles through the works of Gantert, Kim, and Ramanan \cite{GKR2017} and Kabluchko, Prochno, and Thäle \cite{KPT2021}. While universality results are mathematically elegant and also quite powerful, from this perspective, high-dimensional convex bodies cannot be distinguished, for instance, by their lower-dimensional random projections. However, the latter play a fundamental role, e.g., in dimension reduction. Large and moderate deviation principles however are known to be typically non-universal and distribution-dependent in that they are parametric in speed and/or rate function. In this sense, it is reasonable to expect that the large and moderate deviation behavior of, for instance, a random projection of a convex body, depends in a more subtle way on the geometry of the underlying body. The goal to precisely identify how a random projection can encode distinct distributional information about the original vector, triggered an intensive research activity and a whole variety of large and moderate deviation principles have been obtained in the frameworks of Asymptotic Geometric Analysis and High-Dimensional Probability since \cite{alonso2018large,APT_KLS,FP2024,GKR2016,KPT2020_sanov,K2021_sharp,KT2022_weighted,kim2018conditional,kim2021large,LR2024,LRX2023}. Moreover, new tools, such as maximum entropy considerations and elements from log-potential theory, emerged from the theory of large deviations and the related field of statistical mechanics and turned out to be quite powerful, leading to a number of breakthrough results \cite{BW2023,DFGZ2023,JP2023_maxwell,JKP2024,KP2021,KPS2025,KPT2020,KLR2022}. For the current state of the art we refer to the recent survey \cite{Prochno2024} on the large and moderate deviations approach in geometric functional analysis. 
	
The motivation for the present paper stems at first from a result obtained by Paouris, Pivovarov, and Zinn who proved the following central limit theorem in \cite{paouris2014central}. For $N\in\N$ let $\B_\infty^N=[-1,1]^N$ be the unit $N$-dimensional cube and assume that $m\in\N$ with $m\leq N$. Denote by $\mathbb{G}_{m,N}$ the Grassmannian manifold of $m$-dimen\-sional linear subspaces of $\R^N$, which is equipped with the Haar probability measure $\mu_{m,N}$, and suppose that $(E_N)_{N\geq m}$ is a sequence of independent random subspaces such that $E_N\in \mathbb{G}_{m,N}$ has distribution $\mu_{m,N}$. Define for each $N\in\N$ a random variable 
  \[
    Z_N := \vol_m\big(\B_\infty^N\big| E_N \big),
  \]
where $\B_\infty^N\big| E_N$ denotes the orthogonal projection of $B_\infty^N$ onto the linear subspace $E_N$. Then
  \[
    \frac{Z_N-\E[Z_N]}{\sqrt{\V[Z_N]}} \xrightarrow[N\to\infty]{\dint} \mathcal{N}(0,1),
  \]  
where $\stackrel{\dint}{\rightarrow}$ denotes convergence in distribution. At its core, the proof is based on the Gaussian projection formula
  \[
    \vol_m(G \B_\infty^N) = \det (GG^*)^{\frac{1}{2}}\vol_m\big(\B_\infty^N\big| E_N \big)
  \]
 with a standard Gaussian matrix $G\in \R^{m\times N}$ having independent columns $g_1,\dots,g_N$ ($G^*$ denotes the transpose of $G$) and the fact that $G B_\infty^N$ is a random zonotope (i.e., the Minkowski sum of the random line segments $[-g_i,g_i]$). The latter allows a representation of the volume of the Gaussian projection essentially in terms of a $U$-statistic (or a special case of Minkowski's theorem on mixed volumes of convex sets, to put it in convex geometry parlance), which assures that one can apply Vitale's central limit theorem for Minkowski sums of random convex sets \cite{Vitale1987}; one should note, however, that both the Gaussian projection as well as the determinantal part contribute to the asymptotic normality, which requires a careful and delicate analysis (we refer to \cite{paouris2014central} for more information). 
 
 Another motivation for this paper comes from the work of Adamczak, Pivovarov, and Simanjuntak \cite{AdamczakPivovarovSimanjuntak}. In contrast to what has been discussed before, these authors study for general $\ell_p$-balls in $\R^N$ with $0< p\leq\infty$ limit theorems for volumes of random hyperplane sections; see Section \ref{sec:results} for formal definitions. More precisely, if $\B_p^N$ denotes the $\ell_p$-ball in $\R^N$ and if $(H_N)_{N\geq 1}$ is a sequence of independent random elements of the Grassmannian $\mathbb{G}_{N-1,N}$ of hyperplanes in $\R^N$, they show that there is a sequence $(c_N)_{N\geq 1}$ such that
   \[
     N^{3/2} \Bigg( \frac{ \vol_{N-1} (\B_p^N \cap H_N)}{\vol_{N-1}(\B_p^{N-1})}-c_N \Bigg) \xrightarrow[N\to\infty]{\dint} \xi    
   \]
where $\xi \sim \mathcal N(0,\sigma_p^2)$ and both, $(c_N)_{N\geq 1}$ as well as $\sigma_p^2$, are given explicitly (see  \cite[Theorem 3.4]{AdamczakPivovarovSimanjuntak}). A similar result is also obtained for the volumes of the sections with a random subspace of a fixed codimension $m\in\N$ if $0<p<2$ (see \cite[Theorem 3.1]{AdamczakPivovarovSimanjuntak}). The proof of this result is based on probabilistic formulas for volumes of sections of convex bodies developed by Nayar and Tkocz \cite{NayarTrocz} and Chasapis, Nayar, and Tkocz \cite{ChasapisNayarTrocz}. They in turn allow a representation of the volume of $\B_p^N\cap H_n$ in terms of determinants involving auxiliary random variables. Together with a novel reverse H\"older-type inequality, this eventually paves once again the way to applying limit theorems for classical U-statistics. 
 
\medspace

Its is immediate from the discussion above that the proof in \cite{paouris2014central} is tailored for the $\ell_\infty$-structure of the cube and cannot be used to obtain a similar central limit result for general $\ell_p$-balls. Similarly, it seems that the approach taken in \cite{AdamczakPivovarovSimanjuntak} cannot be modified to deal with random sections of fixed dimension, for example. In particular, complementing the central limit theorem for the projections of the cube (and more generally for $\ell_p$-balls) by moderate and large deviation principles for random projections and sections requires a completely different approach with different tools. The aim of this paper is to provide a complete picture in a 4-fold way:
\begin{itemize}
\item[(i)] We consider simultaneously random projections and random sections with random linear subspaces of fixed dimension. 
\item[(ii)] We cover the full range of limit theorems from the central limit perspective over moderate deviation principles to large deviation principles.
\item[(iii)] We study limit theorems for random sections and projections of the general class of $\ell_p^N$-balls.
\item[(iv)] We investigate general geometric functionals of random sections and projections, which includes the volume as a special case. As some of our functionals are related to convex bodies, the convexity of the sections and projections is relevant. This explains the restriction to $p\geq 1$ as opposed to $p>0$ in \cite{AdamczakPivovarovSimanjuntak}.
\end{itemize}

	\section{Main results} \label{sec:results} 
	
	We shall now continue with the presentation of the main results of this paper. The first subsection is devoted to the limit theorems for the volume of random projections and sections, while in the second we present the more general theorems we prove in this paper and from which we actually deduce the result of the first subsection as corollaries.
	
	In the following presentation we use standard notation such as $\|\,\cdot\,\|_p$ for $p$-norm and denote by $\B_p^N$ the unit ball of the normed space $\ell_p^N:=(\R^N,\|\,\cdot\,\|_p)$; we only consider $p\in[1,\infty]$. If $p=2$, then we simply write $\|\,\cdot\,\|:=\|\,\cdot\,\|_2$ and denote by $\mathbb{S}^{N-1}:=\big\{x\in\R^N \,:\, \|x\|=1\big\}$ the Euclidean unit sphere in $\R^N$. For the rest of this paper, we denote by $q\in[1,\infty]$ the H\"older conjugate of $p\in[1,\infty]$, which is defined by the relation$\frac{1}{p}+\frac{1}{q} = 1$; we apply the convention that $\frac{1}{\infty} := 0$. For the Grassmannian manifold consisting of all $m$-dimensional linear subspaces in $\R^N$ we write $\mathbb{G}_{m,N}$  and we denote the unique rotation-invariant Haar probability measure it carries by $\mu_{N,m}$ (also referred to as the uniform distribution on $\mathbb{G}_{m,N}$). For a convex body $C\subseteq \R^m$ (i.e., compact, convex set with non-empty interior), $C\proj E$ denotes the orthogonal projection of $C$ onto $E$, $\vol_m$ is just $m$-dimensional Lebesgue measure, and $\kappa_m := \vol_m(\B_2^m) = \frac{\pi^{\frac{m}{2}}}{\Gamma(1+\frac{m}{2})}$. A Gaussian distribution with expectation $a\in\R$ and variance $b^2>0$ will be denoted by $\mathcal N(a,b^2)$. For other unexplained notation or notions we refer to Section \ref{sec:prelim and notation}.

\subsection{Statement of the main theorems}
	
In all of the following statements $\prosec$ is to be understood as follows: for $m\in\N$, an integer $N\geq m$, a subspace $E \in \mathbb{G}_{m,N}$, and a convex body $C \subseteq \R^N$,
  \[
    C \prosec E :=
    \begin{cases}
     C \big| E &: \quad \text{orthogonal projection of $C$ onto $E$} \cr
     C \cap E& : \quad \text{section of $C$ with $E$},
    \end{cases}
  \]
i.e., the symbol stands for for an orthogonal projection or a section. 

We begin with the central limit theorem (CLT) for the volume of random orthogonal projections and sections of $\ell_p^N$-balls.
	
	\begin{thmalpha}[CLT -- volume of projections and sections of $\ell_p^N$-balls]\label{thm: clt for ellp balls 2}
		Fix $m\in\N$ and $p\in (1,\infty)\setminus \{2\}$. For each $N\geq m$, let $E_{N}\in \mathbb{G}_{m,N}$ be a random subspace with distribution $\mu_{m,N}$.
		Then, for some $\sigma_\prosec >0$ and $\mu_\prosec \in \R$, we have
		\begin{align*}
			\sqrt{N}\Big( \vol_m \big(N^{\frac{1}{p}-\frac{1}{2}}(\mathbb{B}_p^N \prosec E_N)\big) -\mu_\prosec \Big)\xrightarrow[N\to\infty]{\dint} Z,
		\end{align*}
		where $Z \sim \mathcal{N}(0,\sigma_\prosec^2)$. 
		For $\prosec = \big|$ the result extends to the case $p=\infty$, and for $\prosec = \cap$ to the case $p=1$.
	\end{thmalpha}
	
	\begin{rmk}
		We can provide exact values for $\mu_\prosec$ and $\sigma_\prosec$, namely 
		\begin{align*}
			\mu_\prosec &= \begin{cases}
				\frac{2^{m/2} \pi ^{\frac{m (q-1)}{2 q}} \Gamma \left(\frac{q+1}{2}\right)^{m/q}}{\Gamma \left(\frac{m}{2}+1\right)} &: \quad \prosec = \big|\\
				\noalign{\vskip9pt}
				\frac{2^{-\frac{m}{2}} \pi ^{\frac{m (p+1)}{2 p}} \Gamma \left(\frac{p+1}{2}\right)^{-\frac{m}{p}}}{\Gamma \left(\frac{m}{2}+1\right)} &:\quad \prosec = \cap
			\end{cases}
			\intertext{and}
			\sigma_\prosec^2 &= \begin{cases}
				\frac{m \pi^{\frac{qm-m}{q}}\big(2^{q/2}\Gamma((q+1)/2)\big)^{\frac{2m}{q}}\big(4\Gamma(1+m/2)\Gamma(m/2 +q) - (2m+q^2)\Gamma((m+q)/2)^2 \big)}{2q^2\Gamma(1+m/2)^2\Gamma((m+q)/2)^2} &: \quad \prosec = \big|\\
				\noalign{\vskip9pt}
				\frac{m \pi^{(pm+m)/p}\big(2^{p/2}\Gamma((p+1)/2)\big)^{-\frac{2m}{p}}\big(4\Gamma(1+m/2)\Gamma(m/2 +p) - (2m+p^2)\Gamma((m+p)/2)^2 \big)}{2p^2\Gamma(1+m/2)^2\Gamma((m+p)/2)^2} &: \quad \prosec = \cap.
			\end{cases}
		\end{align*} 
	In the formulas above, $q$ denotes the Hölder conjugate of $p$. We refer to Figures \ref{fig:varvol_q} and \ref{fig:section_varvol_q} for a plot of $\sigma_{|}^2$ as a function of $q$ (for fixed $m\in\N$) and to Figure \ref{fig:varvol_m} and \ref{fig:section_varvol_m} for a plot of $\sigma_{\prosec}^2$ as a function of $m$ (for fixed $q$).
	\end{rmk}
	
	\begin{figure}[ht]
		\centering
		\includegraphics[width=\textwidth]{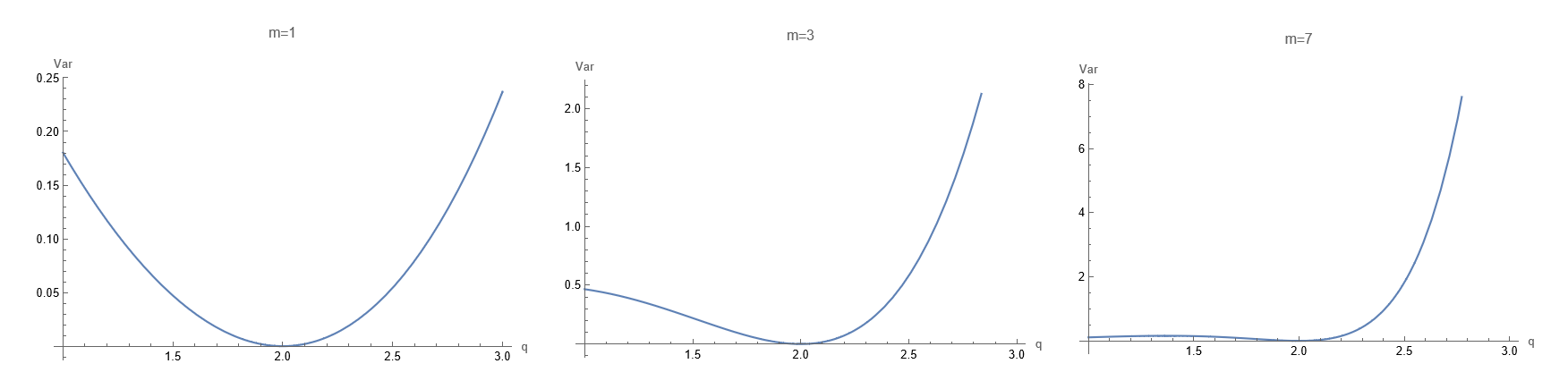}
		\caption{Plot of the asymptotic variance $\sigma_\prosec^2$ in the case of projections in the parameter $q$ for different values of $m$. }
		\label{fig:varvol_q}
	\end{figure}
	
	\begin{figure}[ht]
		\centering
		\includegraphics[width=\textwidth]{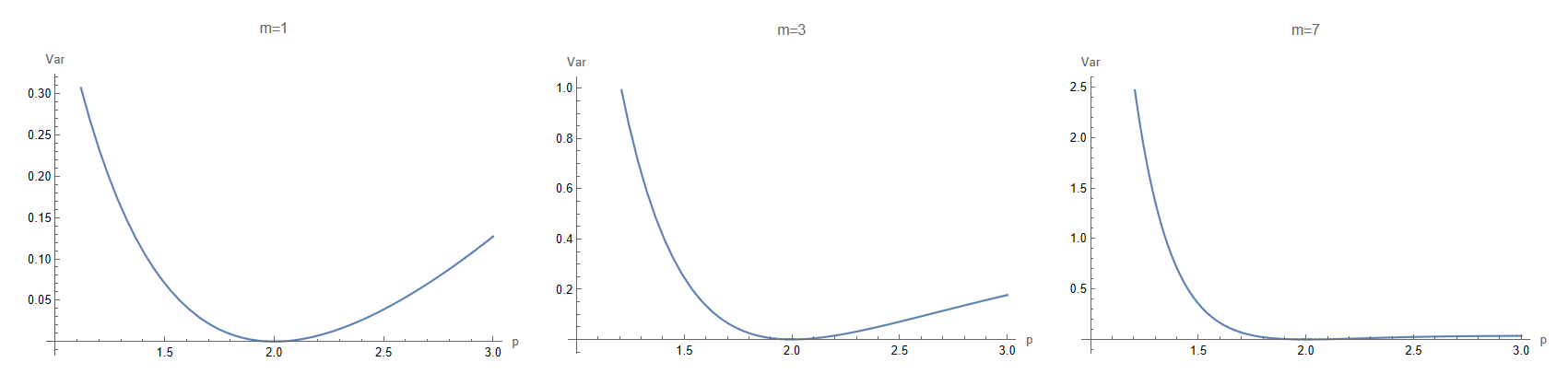}
		\caption{Plot of the asymptotic variance $\sigma_\prosec^2$ in the case of sections in the parameter $p$ for different values of $m$.}
		\label{fig:section_varvol_q}
	\end{figure}
	
	Clearly the plots in Figure \ref{fig:varvol_q} and Figure \ref{fig:section_varvol_q} show that the variance is degenerate in the case $q=p=2$, but also that for large subspace dimensions $m\in\N$ the variance for parameters $q\in(1,2)$, i.e., for $p\in(2,\infty)$, is ``close'' to being zero.  This is in line with the central limit theorem \cite[Theorem 3.4]{AdamczakPivovarovSimanjuntak} for volumes of sections of codimension one, in which a much stronger rescaling of order $N^{3/2}$ is required in order to arrive at a non-degenerate limit distribution.
	 
	\noindent As already mentioned in the introduction, the special case $\prosec = \big|$ and $p= \infty$ was obtained by Paouris, Pivovarov, and Zinn in \cite[Theorem 1.1]{paouris2014central}. Using our notation, they showed that the cube satisfies the central limit theorem 
	$$
	\frac{\vol_m(\mathbb{B}_\infty^N \big| E_N)- \E[\vol_m(\mathbb{B}_\infty^N \big| E_N)]}{\V[\vol_m(\mathbb{B}_\infty^N \big| E_N)]} \xrightarrow[N\to\infty]{\dint} \xi
	$$
	for $\xi \sim  \mathcal{N}(0,1)$. 
	However, exact expressions for the normalizing constants are not provided in \cite{paouris2014central}. In this special case, we obtain 
	$$
	\mu_{\big|} = \frac{2^{m/2}}{\Gamma \left(\frac{m}{2}+1\right)} = \Big(\frac{2}{\pi}\Big)^{m/2}\vol_m(\mathbb{B}^m_2)
	$$
	(note that this is consistent with the case $m=d$ in \cite[Corollary 1.3]{kabluchko2021new}) and the asymptotic variance is given by
	$$
	\sigma^2_{\big|} = 2^{m-1} m \left(\frac{4}{\Gamma \left(\frac{m+1}{2}\right)^2}-\frac{2 m+1}{\Gamma \left(\frac{m}{2}+1\right)^2}\right).
	$$
	The value for $\sigma^2_{\big|}$ was not known before to the best of our knowledge. We also remark that the special case $m=1$ and $\prosec = \big|$ was shown in \cite[Corollary 2.6]{kablu19high}. Note that this is consistent with our result, which provides the asymptotic variance
	$$
	\sigma^2_{\big|} = \frac{2 \pi ^{-1/q} \left(2^{q/2} \Gamma \left(\frac{q+1}{2}\right)\right)^{2/q} \left(2 \sqrt{\pi } \Gamma \left(q+\frac{1}{2}\right) - \left(q^2+2\right) \Gamma \left(\frac{q+1}{2}\right)^2\right)}{q^2 \Gamma \left(\frac{q+1}{2}\right)^2}.\\
	$$
	
	We continue with our second main result, a moderate deviation principle (MDP) complementing our CLT in Theorem \ref{thm: clt for ellp balls 2}. On an intuitive level, if a sequence of random variables $(X_N)_{N\in\N}$ satisfies a CLT in the form $\sqrt{N}(X_n-\E X_N)\xrightarrow[N\to\infty]{\dint} \xi$ with $\xi\sim\mathcal{N}(0,1)$, then $(X_N)_{N\in\N}$ satisfies a (full) moderate deviation principle, provided that for all sequences $(\beta_N)_{N\in\N}$ with $\beta_N\to\infty$ and $\beta_N=o(\sqrt{N})$,
		\begin{equation}\label{eq:MDPLDPintuitive}
		\P\Big[{\sqrt{N}\over\beta_N}(X_N-\E X_N)\geq s\Big] \approx \exp\Big(-N I(s)\Big),\qquad s\in\R,
		\end{equation}
		for some function $I(s)$, where $\approx$ stands for asymptotic equivalence up to sub-exponential terms, the formal definition will follow in Section \ref{sec:prelim and notation}. It is often the case that $I(s)$ is a quadratic function in the variable $s$, which reflects the Gaussian tail behavior. Note that for random orthogonal projections the next theorem generalizes the moderate deviation principle of \cite{kablu19high} when $m=1$.
	
	\begin{thmalpha}[MDP -- volume of projections and sections of $\ell_p$-balls]\label{thm: MDP for volume of ell_p balls}
		Fix $m\in\N$ and let $(\beta_N)_{N\in\N}\in\R^{\N}$ be a sequence with $\beta_N\to\infty$ and $\beta_N = o(\sqrt{N})$. If $\prosec = \big|$, let $p\in(2,\infty]$ and if $\prosec = \cap$, let $p\in [1,2)$. For each $N\geq m$, let $E_{N}\in \mathbb{G}_{m,N}$ be a random subspace with distribution $\mu_{m,N}$. Then the sequence
		$$
		\left(\frac{\sqrt{N}}{\beta_N}\Big(\vol_m\big(N^{\frac{1}{p}-\frac{1}{2}} (\mathbb{B}_p^N \prosec E_N)\big) -  \mu_\prosec \Big)\right)_{N\in\N}
		$$
		satisfies an MDP with speed $\beta_N^2$ and rate function $I_{\mathrm{MDP}}^\prosec$ given in Equation \eqref{eq: MDP rate function volume} below (see also Remark \ref{rmk: MDP rate function}).
	\end{thmalpha}
	
	\begin{figure}[ht]
		\centering
		\includegraphics[width=\textwidth]{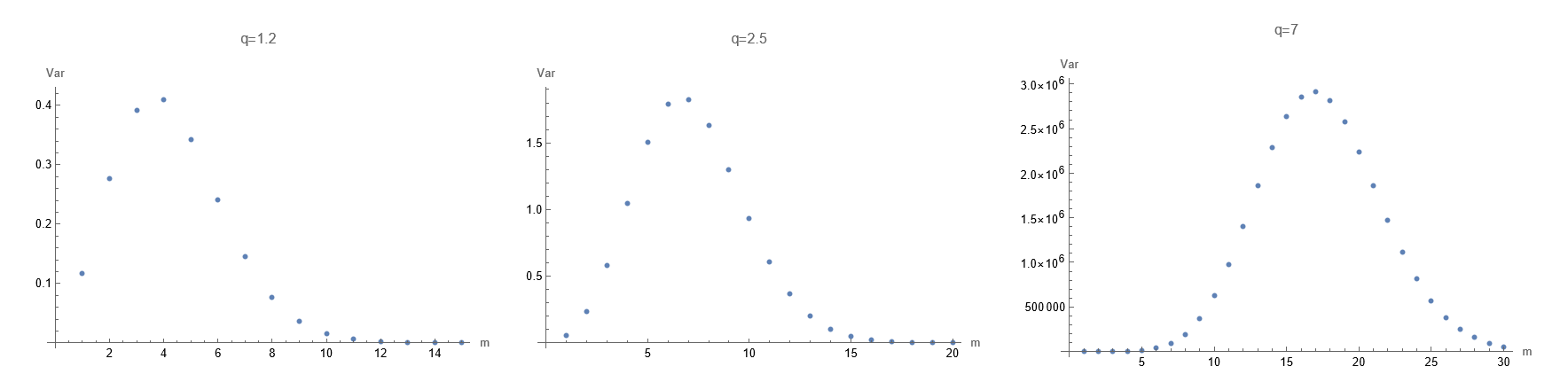}
		\caption{Plot of the asymptotic variance $\sigma_\prosec^2$ in the case of projections in the parameter $m$ for different values of $q$. }
		\label{fig:varvol_m}
	\end{figure}
	
	\begin{figure}[ht]
		\centering
		\includegraphics[width=\textwidth]{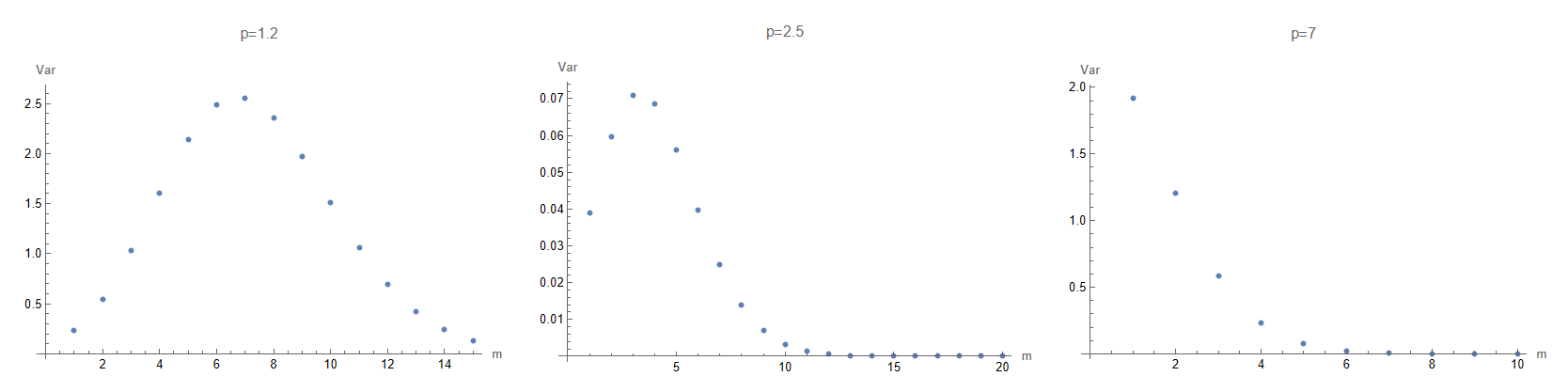}
		\caption{Plot of the asymptotic variance $\sigma_\prosec^2$ in the case of sections in the parameter $m$ for different values of $p$.}
		\label{fig:section_varvol_m}
	\end{figure}

	\noindent On a formal level a large deviation principle (LDP) is the same as a moderate deviation principle as in \eqref{eq:MDPLDPintuitive}, but with $\beta_N=\sqrt{N}$. To present our third main result, an LDP for the volume of projections and sections of $\ell_p$-balls, let us define
	$$
	F_k(\mu) := \Big(\int_{\R^m} |\langle x,u_1\rangle|^q \dint\mu(x)\Big)^{\frac{1}{q}},\ldots,\int_{\R^m} |\langle x,u_k\rangle|^q \dint\mu(x)\Big)^{\frac{1}{q}} \Big) \in \R^k
	$$
	with $\{u_1,u_2,\ldots,u_k\}=\mathbb{S}^{m-1}\cap\Q^m$, where $k\in\N$ and $\mu$ is an element in the space $\mathscr{M}_1(\R^m)$ of Borel probability measures on $\R^m$. Let $\gamma$ denote the standard Gaussian measure on $\R$ and let us write $\gamma^{\otimes m}$ for its $m$-fold product measure on $\R^m$. For $\mu,\nu\in\mathscr{M}_1(\mathbb{R}^m)$, the relative entropy or Kullback--Leibler divergence $H(\nu|\mu)$ is defined as
	\[
	  H(\nu|\mu) :=
	  \begin{cases}
		\int_{\R^m} \log\left(\frac{\dint\nu}{\dint\mu}\right) \dint\nu &:\quad \nu \ll \mu \\
		+\infty &:\quad \text{otherwise,}
	  \end{cases}
	\]
where $\nu \ll \mu$ denotes absolute continuity of $\nu$ with respect to $\mu$ and $\frac{\dint\nu}{\dint\mu}$ the Radon--Nikodym derivative of $\nu$ with respect to $\mu$. Finally, we denote by 
	\begin{align}\label{eq: covariance operaotor C}
		\mathscr{C}:\mathscr{M}_1(\R^m)\to\R^{m\times m},\quad \nu\mapsto \mathscr{C}(\nu) := \int_{\R^m} [x \otimes x] \,\dint \nu(x)
	\end{align}
	the covariance map, where $x\otimes y = (x_jy_i)_{i,j=1}^m$ for $x,y\in\R^m$. 
	We denote by $(C(\mathbb{X}), \|\,\cdot\,\|_\infty)$ the space of real valued continuous functions on the compact topological space $\mathbb{X}$ equipped with the uniform norm $\|f\|_\infty := \sup_{x\in\mathbb{X}}|f(x)|$.
	
	\begin{thmalpha}[LDP -- volume of projections and sections of $\ell_p$-balls]\label{thm: LDP for volume of ell_p balls}
		Fix $m\in\N$. If $\prosec = \big|$ let $p\in(2,\infty]$ and if $\prosec = \cap$ let $p\in [1,2)$. For each $N\geq m$, let $E_{N}\in \mathbb{G}_{m,N}$ be a random subspace with distribution $\mu_{m,N}$. Then the sequence
		$$
		\Big(\vol_m\big(N^{\frac{1}{2}-\frac{1}{q}} (\mathbb{B}_p^N\prosec E_N)\big)\Big)_{N\in\N}
		$$
		satisfies an LDP with speed $N$ and rate function $I_{\mathrm{LDP}}^\prosec: [0,\infty) \to [0,\infty]$,
		$$
		I_{\mathrm{LDP}}^\prosec(x) := 
		  \begin{cases} \inf\Big\{I^\prime_{LDP}(h(C,\,\cdot\,))\,:\,  \vol(C) = x\Big\} & : \quad  \prosec = \big| \\ \inf\Big\{I^\prime_{LDP}(f) \,:\, \int_\Sm f^{-k/p}(t)\dint t = x\Big\} & : \quad \prosec = \cap ,
		  \end{cases}
		$$
		where 
		$I^\prime_{\mathrm{LDP}}:(C(\mathbb{S}^{m-1}),\|\,\cdot\,\|_\infty)\to[0,\infty]$ is given by
		\begin{align*}
			I^\prime_{LDP}(f) := \sup_{k\in\N}\inf\Big\{ I_{\mathrm{Stiefel}}(\mu)\,:\, (f^q(u_1),\ldots,f^q(u_k)) = F_k(\mu)\Big\},
		\end{align*}
	 where $q$ is the Hölder-conjugate of $p$ and $I_{\mathrm{Stiefel}}:\mathscr{M}_1(\R^m)\to [0,\infty]$ is the strictly convex rate function
		$$
		I_{\mathrm{Stiefel}}(\nu) := \begin{cases}
			H(\nu|\gamma^{\otimes m}) + \frac{1}{2}\operatorname{Trace}(\operatorname{Id}_m - \mathscr{C}(\nu)) & : \quad \operatorname{Id}_m - \mathscr{C}(\nu) \text{ positive semidefinite },\\
			+\infty &: \quad \text{ otherwise. }
		\end{cases}
		$$
	\end{thmalpha}
	
	\subsection{Functional results}
	The main theorems presented in the previous section are, in fact, corollaries of more general theorems that we prove. To see how they relate to the main results for the volume, we refer to Section \ref{sec: From functional limit theorems to the volume}.
	
	\noindent In order to state the results, we require some additional notation. We shall denote by $M^*$ the transpose of a matrix $M$. For $N\ge m$ the Stiefel manifold 
	$$
	\mathbb{V}_{m,N} := \Big\{M\in\R^{m\times N}: MM^* = \operatorname{Id}_m\Big\}
	$$
	carries a unique Haar measure, also called uniform distribution, denoted by $\operatorname{Unif}(\mathbb{V}_{m,N})$ which is invariant under left and right multiplication with elements from the orthogonal group. Further, elements from the Stiefel manifold correspond, up to an orthogonal transformation, to an orthogonal projection (see for instance \cite[Proposition 2 and its proof]{paouris2013small}). In more detail, we have $VC=V{\big|_E} (C\big| E)$, where $E=\operatorname{Range}(V^*)$. Then, if $x\in E$, we can write $x=\sum_{i=1}^m\lambda_i v_i$ for some $\lambda_i$, where $v_i$ are the orthonormal rows of $V$. Observe $V{\big|_E} x = (\lambda_1,\ldots,\lambda_m)$ and hence $V{\big|_E}$ can interpreted as a function that maps the orthonormal basis $v_1,\ldots,v_m$ onto the standard basis of $\R^m$. In particular, for any convex body $C\subseteq\R^N$, we have that
	\[
	  \vol_m(VC) \overset{d}{=} \vol_m(C\big| E),
	\]
	where $V\sim \operatorname{Unif}(\mathbb{V}_{m,N})$ and $E\sim \mu_{m,N}$. The support function of {a convex body $C\subseteq \R^m$} is given by
	\[
	  h(C,u) := \sup_{x\in C}\langle x,u\rangle, \quad u\in \Sm.
	\]
	If $0\in C$, then the function $\rho(C,\,\cdot\,):\Sm \to [0,\infty)$,
	\[
	  \rho(C,u) := \sup\big\{r\in [0,\infty) \,:\, ru \in C\big\} 
	\]
	is called the radial function of $C$. Both, the support and radial function, determine the convex body $C$ uniquely. The random functional sequences in the following theorems are transformations of radial or support functions of the projected or intersected $\ell_p^N$-balls; see Lemma \ref{lem: support function representation} and Lemma \ref{lem: section formula}. 
	{Let us also recall the following representation for the volume of a convex body in terms of its support function, which can be found, e.g., in \cite[p. 151]{SchW2008}. It says that for any convex body $C\subseteq \R^m$ with $0\in C$,
	\begin{equation}\label{eq:volume integration formula radial function}
	  \vol_m(C) = \kappa_m \int_{\Sm} \rho(C,u)^m \, \sigma(\dint u),
	\end{equation}
where we write $\sigma:=\sigma_{m-1}$ for the normalized spherical Lebesgue measure on $\Sm$.}

	We start with the following invariance principle dealing with the weak convergence of random continuous functions on $\Sm$ that are sums over functions of columns of the Stiefel manifold. In the following ${}_2 F_1$ denotes the Gaussian hypergeometric function (see \cite[Chapter 15]{lozier2003nist}) and for random variables $\xi,\xi_N$, $N\in\N$, taking values in some metric space, we shall write $\xi_N\xRightarrow[N\to\infty]{} \xi$ or $\xi_N \xrightarrow{d} \xi$ to denote weak convergence of the corresponding  distributions. Moreover, let us recall that $G(u)$, $u\in\Sm$, is said to be a Gaussian process indexed by $\Sm$ if and only if for any finite collection $u_1,\dots,u_k\in\Sm$, $(G(u_1),\dots,G(u_k))$ is a multivariate Gaussian random variable.
	
	\begin{thm}[Functional CLT for sums over the Stiefel manifold]\label{thm: Functional CLT}
		Let $V_N=(v_1,\dots,v_N)\in\R^{m\times N}$ be uniformly distributed in $\mathbb{V}_{m,N}$ for each $N\ge m$. Let $q\in [1,\infty)\setminus \{2\}$. Then we have
	\[
           \frac{1}{\sqrt{N}}\sum_{i=1}^N  \Bigg(\Big|\langle \sqrt{N}v_i, (\,\cdot\,)\rangle\Big|^q - \E\Big[|\langle \sqrt{N}v_i, (\,\cdot\,)\rangle|^q\Big]\Bigg) \xRightarrow[N\to\infty]{}  Z(\,\cdot\,),
	\]
where 
the random variables take values in $(C(\Sm), \|\,\cdot\,\|_\infty)$, and $Z$ is a centered Gaussian process indexed by $\Sm$ whose covariance structure is given by
		\begin{align*}
			&\E[Z(u)Z(v)] = \E[|\langle g,u\rangle\langle g,v\rangle|^q] - \E[|g|^q]^2\Big(1 + \frac{q^2}{2}\sum_{k=1}^m v_k^2 u_k^2 + q^2\sum_{k=1}^m \sum_{\ell = k+1}^m u_k u_\ell v_k v_\ell \Big)\\
			&= \frac{2^q\Gamma(\frac{q+1}{2})^2}{\pi} {}_2 F_1(-q/2,-q/2, 1/2, \langle u,v\rangle^2) - \frac{2^q \Gamma(\frac{q+1}{2})^2}{\pi}\Big(1 + \frac{q^2}{2}\sum_{k=1}^m v_k^2 u_k^2 + q^2\sum_{k=1}^m \sum_{\ell = k+1}^m u_k u_\ell v_k v_\ell \Big)
		\end{align*}
		for $u,v\in \Sm$ with $u\not=v$, and
		$$
		\E[Z(u)^2]= \frac{2^q \Gamma(\frac{1}{2}+q)}{\sqrt{\pi}} - \frac{2^q \Gamma(\frac{q+1}{2})^2}{\pi}\Big(1 + \frac{q^2}{2}\sum_{k=1}^m u_k^4 + q^2\sum_{k=1}^m \sum_{\ell = k+1}^m u_k^2 u_\ell^2  \Big)
		$$ 
		for $u\in\Sm$.
		The same holds true if we replace $\E[|\langle \sqrt{N}v_i, (\,\cdot\,)\rangle|^q]$ in the process by its asymptotic equivalent $\E[|g|^q]$, where $g\sim \mathcal{N}(0,1)$.
	\end{thm}

	\noindent We also state the corresponding functional MDP and LDP results.
	
	\begin{thm}[Functional MDP for sums over the Stiefel manifold]\label{thm: Functional MDP}
		Let $V_N=(v_1,\dots,v_N)\in\R^{m\times N}$ be uniformly distributed in $\mathbb{V}_{m,N}$ for each $N\ge m$. Let $q \in [1,2)$. Let $\beta_N$ be a sequence of real numbers with $\beta_N\to\infty$ and $\beta_N = o(\sqrt{N})$. Then the sequence 
		$$
		\Bigg( \frac{1}{\beta_N\sqrt{N}}\sum_{i=1}^N  \Bigg(\Big|\langle \sqrt{N}v_i, (\,\cdot\,)\rangle\Big|^q - \E\Big[|\langle \sqrt{N}v_i, (\,\cdot\,)\rangle|^q\Big]\Bigg)\Bigg)_{N\in\N}
		$$
		satisfies an MDP with speed $\beta_N^2$ on $(C(\mathbb{S}^{m-1}),\|\cdot\|_\infty)$ and good rate function given in Remark \ref{rmk: MDP rate function}. This remains true if we replace $\E[|\langle \sqrt{N}v_i, (\,\cdot\,)\rangle|^q]$ in the process with $\E[|g|^q], g\sim \mathcal{N}(0,1)$.
	\end{thm}
	
	\begin{thm}[Functional LDP for sums over the Stiefel manifold]\label{thm: Functional LDP}
		Let $V_N=(v_1,\dots,v_N)\in\R^{m\times N}$ be uniformly distributed in $\mathbb{V}_{m,N}$ for every $N\ge m$. Then, for every $q\in [1,2)$, the sequence 
		$$
		  \Bigg(\frac{1}{N}\sum_{i=1}^N  \Big|\langle \sqrt{N}v_i, (\,\cdot\,)\rangle\Big|^q \Bigg)_{N\in\N}
		$$
		satisfies an LDP with speed $N$ on $(C(\mathbb{S}^{m-1}),\|\,\cdot\,\|_\infty)$ and rate function $I^\prime_{LDP}:(C(\mathbb{S}^{m-1}),\|\,\cdot\,\|_\infty)\to[0,\infty]$ given by
		\begin{align*}
			I_{\mathrm{LDP}}^\prime(f) = \sup_{k\in\N} \, \inf \Big\{ I_{\mathrm{Stiefel}}(\mu) \,:\, (f^q(u_1),\ldots,f^q(u_k)) = F_k(\mu)\Big\},
		\end{align*}
		where $I_{\mathrm{Stiefel}}:\mathscr{M}_1(\R^m)\to [0,\infty]$ is the strictly convex rate function 
		$$
		I_{\mathrm{Stiefel}}(\nu) := \begin{cases}
			H(\nu|\gamma^{\otimes m}) + \frac{1}{2}\operatorname{Trace}\big(\operatorname{Id}_m - \mathscr{C}(\nu)\big) & : \quad \operatorname{Id}_m - \mathscr{C}(\nu) \text{ is positive semidefinite },\\
			+\infty & : \quad\text{ otherwise. }
		\end{cases}
		$$
	\end{thm}
	
	The following provides an interpretation for all of the above convergence results. Let $E_N\in \mathbb{G}_{m,N}$ be distributed according to the uniform distribution $\mu_{m,N}$ and let $\mathbb{B}_2^m(E_N)$ be the $m$-dimensional Euclidean unit ball in $E_N$. Then we have the following almost sure convergence in the Hausdorff distance (for background information on the space of convex bodies and the notion of Hausdorff distance, we refer to \cite[Chapter 1.8]{schneider2014convex}): 
	\begin{align}\label{eq: Hausdorff conv to ball proj}
		d_H\Big(N^{\frac{1}{p}-\frac{1}{2}}(\mathbb{B}_p^N \big| E_N), \frac{\sqrt{2}}{\pi^{\frac{1}{2q}}} \Gamma\big(\frac{q+1}{2}\big)^{1/q}\mathbb{B}_2^m(E_N)\Big) \xrightarrow{a.s.} 0
	\end{align}
	for $p\in (1,\infty]$ and
	\begin{align}\label{eq: Hausdorff conv to ball sec}
		d_H\Big(N^{\frac{1}{p}-\frac{1}{2}}(\mathbb{B}_p^N\cap E_N), \frac{\pi^{\frac{1}{2p}}}{\sqrt{2}}\frac{1}{\Gamma(\frac{p+1}{2})^{1/p}}\mathbb{B}_2^m(E_N)\Big) \xrightarrow{a.s.} 0,
	\end{align}
	for $p\in [1,\infty)$. These limit theorems show that projections and sections of appropriately rescaled $\ell_p^N$-balls in high dimensions are close to Euclidean balls of appropriate radii within the corresponding subspaces. The result for the volume of sections was already shown in \cite[Section 5]{koldobsky2000average} and the case of projections and $p=\infty$ can be found in \cite[Theorem 1.1]{kabluchko2021new}. This convergence can also be deduced using the Dvoretzky--Milman theorem \cite{Figiel1977dim} as is shown, for instance, in \cite[Corollary 7.24]{brannan2021alice}). To see this, one needs to combine the value for the Dvoretzky dimension for $\ell_p^N$-balls (see \cite[Example 3.1]{Figiel1977dim} or \cite[Theorem 7.31 and Equation (5.63)]{brannan2021alice}) with the Borel--Cantelli lemma exactly as in \cite[Remark 1.2]{kabluchko2021new}. We provide an alternative self-contained proof for this convergence in Section \ref{sec: proofs} that is aligned with our methods of proof used throughout this work. In this light, our main results can be seen as characterizations of the above convergences in different ways, including the speed of convergence and fluctuations. 

	\section{Preliminaries and more notation}\label{sec:prelim and notation}
	
	In this section, we provide a number of definitions, notation, and standard results that appear throughout the paper.
	
	\subsection{Notation related to $\ell_p$-spaces}

  For $N\in\mathbb{N}$, $x:=(x_1,\ldots,x_N)\in\R^N$ and $p\in[1,\infty]$, we denote the $\ell_p^N$-norm of $x$ by 
	$$
	\|x\|_p:= \begin{cases}\left(\sum_{i=1}^N \left|x_i\right|^p\right)^{1 / p} & : \quad p \in[1, \infty), \\ \max \left\{\left|x_1\right|,\ldots, \left|x_N \right| \right\} & : \quad p=\infty .\end{cases}
	$$
	The corresponding $\ell_p^N$-spaces are defined as $\ell_p^N := (\R^N, \|\,\cdot\,\|_p)$ and the closed unit ball in $\ell_p^N$ is
	$$
	\mathbb{B}_p^{N}:= \Big\{x\in\R^N \,:\, \|x\|_p\le 1\Big\}.
	$$
	If $p=2$, we typically write $\|\,\cdot\,\| := \|\,\cdot\,\|_2$ and $\mathbb{S}^{N-1}:=\big\{x\in\R^N\,:\,\|x\|=1\big\}$ for the Euclidean unit sphere in $\R^N$. We denote by $q$ the H\"older conjugate of $p$, which is defined by the relation
	$$
	\frac{1}{p}+\frac{1}{q} = 1
	$$
	and where we use the convention that $\frac{1}{\infty} := 0$, implying that $q=1$ is the Hölder conjugate of $p=\infty$.
	
	\subsection{Asymptotic notation and notation from probability theory}
	
	For two sequences $(\alpha_N)_{N\in\N}$ and $(\beta_N)_{N\in\N}$ of real numbers, we write $\alpha_N \sim \beta_N$ if $\frac{\alpha_N}{\beta_N}\to 1$, as $N\to\infty$. For a random element $X$ and a probability measure $\mu$, the notation $X\sim \mu$ means that $X$ is distributed according to $\mu$. Further, we denote $\overset{d}{=}$ and $\xrightarrow{d}$ for equality and convergence in distribution, respectively, and we indicate by $\mathcal{N}(0,\Sigma)$ the centred Gaussian distribution with covariance matrix $\Sigma$ and by $\mathcal{N}(0,\operatorname{Id}_m)$ the standard Gaussian distribution in $\R^m$ where $\operatorname{Id}_m$ is the $m\times m$ identity matrix. The symbol $\sigma$ will always denote  the normalized spherical Lebesgue measure on the Euclidean unit sphere. 
	
	\subsection{Brackground from the theory of large deviations}
	
	For a topological space $\mathbb{X}$, a function $I:\mathbb{X}\to[0,\infty]$ is said to be a rate function if all sublevel sets are closed, i.e., if  $\{x\in\mathbb{X}: f(x)\le y\}$ is closed in $\mathbb{X}$ for every $y$. A rate function is called good if the sublevel sets are additionally compact. A large deviation principle is defined as follows.
	\begin{df}\label{def: LDP}
		Let $(X_N)_{N\in\N}$ be a sequence of random variables taking values in a topological space $\mathbb{X}$. The sequence is said to satisfy a \emph{large deviation principle (LDP)} with \emph{speed} $s_N$, where $s_N\to\infty$, and \emph{rate function} $I:\mathbb{X}\to [0,\infty]$ if and only if
		\begin{align}
			\liminf_{N\to\infty} \frac{1}{s_N}\log \P[X_N \in O] &\ge -\inf_{x\in O}I(x)\quad\text{ for all open } O\subseteq\mathbb{X},\label{LDP lower bound}\\ 
			\limsup_{N\to\infty} \frac{1}{s_N}\log \P[X_N\in C] &\le -\inf_{x\in C}I(x)\quad\text{ for all closed } C\subseteq\mathbb{X}\label{LDP upper bound}.
		\end{align}
		Let $\beta_N\to\infty$ with $\beta_N = o(\sqrt{N})$. We say that the sequence $(\beta_N X_N)_{N\in\N}$ satisfies a \emph{moderate deviation principle} (MDP) if and only if it satisfies an LDP with speed $\beta_N^2$.	
		\end{df}
	
	We say that $(X_N)_{N\in\N}$ satisfies a \emph{weak LDP} if it satisfies the definition of an LDP with the exception that the upper bound (Equation (\ref{LDP upper bound})) only holds for compact sets instead of all closed sets. A closely related concept is that of exponential tightness. \\
	
		
		\begin{df}
			A sequence of random variables $(X_N)_{N\in\N}$ on a topological space $\mathbb{X}$ is \emph{exponentially tight} with respect to the speed $s_N$ if and only if for every $a>0$ there exists a compact set $K\subseteq \mathbb{X}$ such that
			$$
			\limsup_{N\to\infty} \frac{1}{s_N} \log \P[X_N\in \mathbb{X}\setminus K] \le -a.
			$$
		\end{df}
		If nothing else is mentioned, we assume $s_N = N$. An important property of exponential tightness is that a sequence of random variables satisfies an LDP with speed $s_N$ and good rate function if and only if it is exponentially tight with respect to $s_N$ and it satisfies the corresponding weak LDP (\cite[Lemma 1.2.18]{dembo2009techniques}). Tuples are exponentially tight exactly if each component is exponentially tight as the next lemma shows. 
		
		\begin{lemma}\label{lem: exp tight produkt space}
			Assume that $(X_N)_{N\in\N}$ and $(Y_N)_{N\in\N}$ are exponentially tight sequences of random variables on the metric spaces $\mathbb{X}$ and $\mathbb{Y}$, respectively. Then, the sequence $(X_N, Y_N)_{N\in\N}$ is exponentially tight on $\mathbb{X}\times\mathbb{Y}$ equipped with the product topology.
		\end{lemma}
		\begin{proof} Let $\epsilon>0$. By assumption there exist compact sets $K_1\subseteq \mathbb{X}$ and $K_2 \subseteq \mathbb{Y}$ such that 
			\begin{align*}
				\limsup_{N\to\infty}\frac{1}{s_N}\log \P\big[X_N\in K_1^c\big]&\le -\epsilon,\\
				\limsup_{N\to\infty}\frac{1}{s_N}\log \P\big[Y_N\in K_2^c\big]&\le -\epsilon.
			\end{align*}
			Now, define the set $K = K_1\times K_2 \subseteq \mathbb{X}\times \mathbb{Y}$, which is compact in the product topology (by Tychonoff's theorem, see \cite{Chernoff1992}). We have
			\begin{align*}
				&\limsup_{N\to\infty}\frac{1}{s_N}\log \P\big[(X_N,Y_N) \in K^c\big]\\
				&\le \limsup_{N\to\infty}\frac{1}{s_N}\log\Big( \P\big[(X_N,Y_N) \in (K_1^c, \mathbb{Y})\big] + \P\big[(X_N,Y_N) \in (\mathbb{X}, K_2^c)\big]\Big)\cr
				&= \max\Bigg\{\limsup_{N\to\infty}\frac{1}{s_N}\log \P\big[(X_N,Y_N) \in (K_1^c, \mathbb{Y})\big],\, \limsup_{N\to\infty}\frac{1}{s_N}\log \P\big[(X_N,Y_N) \in (\mathbb{X}, K_2^c)\big] \Bigg\} \cr
				&\le -\epsilon,
			\end{align*}
			by Proposition \cite[Lemma 1.2.15]{dembo2009techniques}, which proves exponential tightness.
		\end{proof}
		
		We present a method that can be used to transfer an LDP for one sequence to another sequence if they are exponentially equivalent.
		
		\begin{proposition}[Exponential equivalence, \text{\cite[Theorem 4.2.13]{dembo2009techniques}}]
			Assume $(X_N)_{N\in\N}$ and $(Y_N)_{N\in\N}$ are sequences of random variables on the metric space $\mathbb{X}$ which is equipped with the metric $d$ and assume that $(X_N)_{N\in\N}$ satisfies an LDP with speed $s_N$ and good rate function $I$. If $(X_N)_{N\in\N}$ and $(Y_N)_{N\in\N}$ are \emph{exponentially equivalent}, i.e.,
			$$
			\limsup_{N\to\infty}\frac{1}{s_N}\log \P\big[d(X_N, Y_N)>\epsilon\big] = -\infty
			$$
			for every $\epsilon>0$, then $(Y_N)_{N\in\N}$ satisfies an LDP with the same speed and the same rate function. 
		\end{proposition}

		The following is a standard tool that allows us to continuously transform an LDP into another.
		
		\begin{proposition}[Contraction principle, \text{\cite[Theorem 4.2.1]{dembo2009techniques}}]\label{prop: contraction principle}
			Let $f:\mathbb{X}\to\mathbb{Y}$ be a continuous function and $\mathbb{X}, \mathbb{Y}$ be Hausdorff topological spaces. Let $(X_{N})_{N\in\N}$ be a sequence of random variables that satisfies an LDP on $\mathbb{X}$ with speed $s_N$ and rate function $I$. Then, the sequence $(f(X_N))_{N\in\N}$ satisfies an LDP on $\mathbb{Y}$ with speed $s_N$ and rate function $J:\mathbb{Y}\to [0,\infty]$,
			$$J(y) = \inf_{x\in f^{-1}(\{y\})}I(x).$$
		\end{proposition}
		
		\noindent By convention, we set $J(y) = \infty$, whenever $f^{-1}(\{y\}) = \emptyset$. If the context is clear, we may write $f^{-1}(y)$ instead of $f^{-1}(\{y\})$.\\
		
		One prominent results in large deviation theory is Cram\'er's theorem. Similar to other limit theorems as the central limit theorems or the laws of large numbers, it characterizes the behavior of means of i.i.d.\- random variables. 
		
		\begin{proposition}[Cram\'er's theorem, \text{\cite[Chapter 24]{kallenberg1997foundations}}]\label{thm: cramer}
			Let $(X_N)_{N\in\N}$ be an i.i.d.\- sequence of random vectors in $\R^n$. Assume that the log-cumulant generating function of $X_1$,
			$$
			\Lambda (\lambda) := \log \E\Big[e^{\langle \lambda,X_1\rangle}\Big], \quad \lambda\in\R^n,
			$$
			is finite in a neighborhood of $0\in \R^n$. Then the sequence $(\frac{1}{N}\sum_{i=1}^N X_i)_{N\in\N}$ satisfies an LDP with speed $N$ and rate function $\Lambda^*$ given by the \emph{Fenchel--Legendre transform} of $\Lambda$, 
			\begin{align}
				\Lambda^*(x) := \sup_{\lambda\in \R^n}\big(\langle \lambda,x\rangle-\Lambda(\lambda)\big), \quad x\in\R^n.\label{eq: Legendre-fenchel}
			\end{align}
		\end{proposition}
		
		The following is the analog for the MDP case.
		
		\begin{proposition}[MDP for empirical means in $\R^m$, \text{\cite[Theorem 3.7.19]{dembo2009techniques}}]\label{prop: MDP in R^m}
			Let $X_1,X_2,\ldots$ be a sequence of i.i.d.~random vectors in $\R^m$ such that $\log \E[e^{\langle \lambda, X_1\rangle}]<\infty$ for all $\lambda$ in some neighborhood around the origin and let $\beta_N$ be a sequence of real numbers such that $\beta_N\to\infty$, $\beta_N = o(\sqrt{N})$. Further, assume $\E[X_1] = 0$ and that the covariance matrix $\mathcal{C}$ of $X_1$ is invertible. Then the sequence 
			$$
			\Big(\frac{1}{\sqrt{N}\beta_N}\sum_{i=1}^N X_i\Big)_{N\in\N}
			$$ 
			satisfies an LDP on $\R^m$ with speed $\beta_N^2$ and rate function $I:\R^m\to[0,\infty]$, $I(x) := \frac{\langle x,\mathcal{C}^{-1}x\rangle}{2}$.
		\end{proposition}
		
		We introduce the notion of projective limits. A \emph{projective system} $\left(\mathbb{Y}_j, p_{i j}\right)_{i \leq j}$ consists of Hausdorff topological spaces $\mathbb{Y}_j, j \in \mathbb{N}$, and continuous mappings $p_{i j}: \mathbb{Y}_j \rightarrow \mathbb{Y}_i$ such that for all $i \leq j \leq k$, we have $p_{i k}=p_{i j} \circ p_{j k}$ and $p_{j j}, j \in \mathbb{N}$ is the identity mapping on $\mathbb{Y}_j$. Then, the projective limit $\mathbb{Y}$ of this system is given by
		$$
		\mathbb{Y}:=\Big\{y=\left(y_j\right)_{j \in \mathbb{N}} \in \prod_{j \in \mathbb{N}} \mathbb{Y}_j \,:\, y_i=p_{i j}\left(y_j\right), \forall i<j\Big\}.
		$$
		We denote by $p_j:\mathbb{Y}\to \mathbb{Y}_j$ for $j\in\N$ the canonical projection, i.e., for $y\in\mathbb{Y}$, we have $p_j(y) = y_j$ in the notation above. The following theorem by Dawson and G\"artner allows inferring an LDP on $\mathbb{Y}$ given an LDP on $\mathbb{Y}_j$ for each $j$.
		
		\begin{proposition}[Dawson and G\"artner, \text{\cite[Theorem 3.3]{dawsont1987large}}]\label{prop: dawson-g\"artner}
			Let $\mathbb{Y}$ be the projective limit of the projective system $\left(\mathbb{Y}_j, p_{i j}\right)_{i \leq j}$. Assume that $\left(X_N\right)_{N \in \mathbb{N}}$ is a sequence of random variables on $\mathbb{Y}$ such that for any $\ell \in \mathbb{N}$, $\left(p_{\ell}(X_N)\right)_{N \in \mathbb{N}}$ satisfies an LDP on $\mathbb{Y}_{\ell}$ with speed $s_N$ and rate function $I_{\ell}$. Then $\left(X_N\right)_{N \in \mathbb{N}}$ satisfies an LDP with speed $s_N$ and rate function $I: \mathbb{Y} \rightarrow[0,+\infty]$ given by
			$$
			I(x):=\sup _{\ell \in \mathbb{N}} I_{\ell}\left(p_{\ell}(x)\right) .
			$$
		\end{proposition}

		\subsection{Hadamard differentiability, the delta method, and a functional CLT}
		
		Let us introduce the notion of Hadamard differentiability, which arises from a concept of directional derivative for maps between Banach spaces; in our situation we shall be interested in the Banach space $(C(\Sm), \|\,\cdot\,\|_\infty)$ of continuous functions on $\Sm$. This concept will appear in an infinite-dimensional version of the delta method.

		\begin{df}
			Let $(\mathbb{X}, \|\,\cdot\,\|_{\mathbb X})$ and $(\mathbb{Y}, \|\,\cdot\,\|_{\mathbb Y})$ be real normed spaces and consider a map $\Phi: \mathbb{X}\to\mathbb{Y}$. We say that $\Phi$ is \emph{Hadamard differentiable} at $x\in \mathbb{X}$ if and only if there exists a map $\dint\Phi(x):\mathbb{X}\to\mathbb{Y}$, which we call the derivative of $\Phi$, such that for all (directions) $h\in\mathbb X$ and all sequences $(h_N)_{N\in\N}\in\mathbb X^\N$ with $h_N\to h$, and all sequences $(t_N)_{N\in\N}\in\R^\N$ with $t_N\to 0$ and $t_N\neq 0$, we have
			\[
			   \Bigg\|\frac{\Phi(x+t_Nh_N) - \Phi(x)}{t_N} - \dint \Phi(x)[h]\Bigg\|_{\mathbb Y} \stackrel{N\to\infty}{\longrightarrow} 0.
			\]			
		\end{df}
				
	
	It follows directly from the definition that the map $\dint\Phi(x)$ is always continuous.
	The next proposition is a generalized delta method.

		\begin{proposition}[Delta method for infinite-dimensional spaces, \text{\cite[Theorem 1]{romisch2004delta}}]\label{prop: delta method infinite dim}
			Let $X_N$ be random variables with values in a normed vector space $\mathbb{X}$ such that 
			$$
			b_N(X_N-\mu) \overset{d}{\longrightarrow} Z
			$$
			for some $\mu\in\mathbb{X}$, some sequence $b_N\to\infty$, and some random variable $Z$ taking values in $\mathbb{X}$. Further, assume that $\phi:\mathbb{X}\to \mathbb{Y}$, where $\mathbb{Y}$ is some other normed vector space, is Hadamard differentiable at $\mu$ with derivative $\dint\Phi(\mu)$. Then we have 
			$$
			b_N(\phi(X_N)-\phi(\mu)) \overset{d}{\longrightarrow} \dint\Phi(\mu)[Z].
			$$ 
		\end{proposition}
		
		\noindent If $Z$ is Gaussian and $\dint\Phi(\mu)$ is linear, then $\dint\Phi(\mu)[Z]$ is also Gaussian (see, e.g., \cite{romisch2004delta}). \\
		
		\noindent A similar delta method exists for large deviations and we only state a special case.
		
		\begin{proposition}[Delta method for large deviation principles, \text{\cite[Theorem 3.1]{gao2011delta}}]\label{prop: delta method for LDPs}
			Let $\mathbb{X}$, $\mathbb{Y}$ be a normed vector spaces, $\mathbb{D}\subseteq \mathbb{X}$ open and $\Phi: \mathbb{D}\to \mathbb{Y}$ be Hadamard differentiable in $\mu$ with Hadamard derivative $\dint\Phi(\mu)$. Let $X_1,X_2,\ldots$ be random variables with values in $\mathbb{D}$. Further, assume that 
			$$
			b_N(X_N - \mu)
			$$
			satisfies an LDP with speed $s_N$ and rate function $I:\mathbb{X}\to [0,\infty]$ for some sequence $b_N \to\infty$. Then 
			$$
			b_N(\Phi(X_N)-\Phi(\mu))
			$$
			satisfies an LDP with speed $s_N$ and rate function $I^\prime:\mathbb{Y}\to[0,\infty]$, 
			$$
			I^\prime(x) = \inf_{\dint\Phi(\mu)[y] = x}I(y).
			$$
		\end{proposition}
		
		\noindent Finally, we need a central limit theorem for random variables having values in function spaces.
		
		\begin{proposition}[Functional central limit theorem, \text{\cite{gine1974central}}]\label{prop: functional CLT}
			Let $(\mathbb{X},d)$ be a compact metric space and let $(X(u))_{u\in \mathbb{X}}$ be a real-valued and centered random process on $\mathbb{X}$ satisfying $\E X(u)^2 < \infty$ for all $u\in \mathbb{X}$. Furthermore, suppose there exist a square-integrable random variable $L$ such that 
			$$
			|X(u) -X(v)| \le L d(u,v),\qquad u,v\in\mathbb{X}.
			$$
			Finally, assume that
			$$
			\int_{0}^1 \log\big(\operatorname{Covering}(\mathbb{X},d,\epsilon)\big)\,\dint \epsilon < \infty,
			$$
			where $\operatorname{Covering}(\mathbb{X},d,\epsilon)$ denotes the covering number of $\mathbb{X}$ with respect to the metric $d$ and radius $\epsilon$. Then there exist a centered Gaussian process $Z=(Z(u))_{u\in\mathbb{X}}$ on $C(\mathbb{X})$ with  $\E[Z(u)Z(v)] = \E[X(u)X(v)]$ for all $u,v\in \mathbb{X}$, and such that for a sequence of independent random elements $X_1,X_2,\ldots$, all having the same distribution as $X$, the convergence in distribution
			$$
			\frac{1}{\sqrt{N}}\sum_{i=1}^N X_i \xrightarrow[N\to\infty]{\dint} Z
			$$
			holds on the space $(C(\mathbb{X}), \|\,\cdot\,\|_\infty)$. 
		\end{proposition}
		
		\section{Proofs}\label{sec: proofs}
		
		In this section we present all proofs. We start in Section \ref{sec: General functional results} with two more general results about functional limit theorems in the context of large deviations and weak convergence. In Section \ref{sec: Volume and support function representation} we present how the functional results relate to the limit theorems for the projected or intersected $\ell_p^N$-balls by giving explicit forms of the support function and a volume representation. We continue in Section \ref{sec: Calculation of Hadamard differentials} with calculations of Hadamard differentials that are needed to perform this transition between functional and volume results. As a last preparation, we calculate various integrals in Section \ref{sec: Computations of moments} that are needed for an explicit form of the variance in the central limit theorem. Finally, the proofs of the functional CLT and MDP follow in Section \ref{sec: CLT and MDP proof} and the functional LDP in Section \ref{sec: LDP proof}. The transition from the functional to the main results is then the content of Section \ref{sec: From functional limit theorems to the volume} concluding the proofs section. More details outlines can be found in the corresponding sections.
		
		\subsection{General functional results in large deviation theory and weak convergence}\label{sec: General functional results}
		
		The following lemma allows us to deduce a weak LDP for a sequence of random continuous functions if each collection of finite evaluation tuples satisfy an LDP. We state it in more generality than needed as it might be of independent interest.
		
		\begin{lemma}\label{lem: LDP for function}
			Let $f_i:\mathbb{X}\to\mathbb{Y}$ be a continuous function from a separable topological space with countable dense subset $\mathbb{U}$ to a Hausdorff topological space for each $i\in \mathbb{T}$ where $\mathbb{T}$ is {another Hausdorff topological space}. Assume that, for each collection $\{u_1,...,u_k\}\subseteq\mathbb{U}$ and a sequence of random variables $(X_N)_{N\in\N}$ taking values in $\mathbb{T}$, the sequence of tuples 
			$$
			\left(f_{X_N}(u_1),...,f_{X_N}(u_k)\right)_{N\in\N}
			$$ 
			satisfies an LDP in $\mathbb{Y}^k$ with speed $s_N$ and rate function $J_k:\mathbb{Y}^k\to[0,\infty]$. Then, the sequence of functions $(f_{X_N})_{N\in\N}$ satisfies a weak LDP on the space $(C(\mathbb{X},\mathbb{Y}), \|\,\cdot\,\|_\infty)$ with speed $s_N$ and rate function 
			$$
			I:\{f_i : i\in \mathbb{T}\}\to[0,\infty],\quad I(f) = \sup_{k\in\N} J_k(f(u_1),...,f(u_k)),
			$$ 
			where $\{u_1,u_2,...\} = \mathbb{U}$.
		\end{lemma}
		
		\begin{proof} Let $\{u_1,u_2,...\} = \mathbb{U}$ be an enumeration of the dense countable subset. Define the projective system $(\mathbb{Y}_{k}, p_{\ell k})_{\ell \le k}$ by
			$$
			\mathbb{Y}_k := \big\{ (f_i(u_1),...,f_i(u_k))\,:\,\enskip i\in \mathbb{T}\big\}
			$$
			for $k\in\N$ with the product topology and for $\ell \le k$ the projection maps 
			$$
			p_{\ell k}(f_i(u_1),...,f_i(u_k)) = (f_i(u_1),...,f_i(u_\ell)).
			$$ 
			Next, define 
			$$
			\mathbb{Y} := \big\{ (f_i(u_1),f_i(u_2),...) \,:\,  i\in \mathbb{T}\big\},
			$$
			also equipped with the product topology and projection maps 
			$$
			p_k:\mathbb{Y}\to\mathbb{Y}_k, \quad p_k(f_i(u_1),f_i(u_2),...) = (f_i(u_1),...,f_i(u_k))
			$$
			for $k\in\N$. These maps are continuous and $p_i(y)=p_{i j}(p_j(y))$ for all $y\in\mathbb{Y}$, $i\le j$ holds. By assumption, $(p_k(f_{X_N}(u_1),f_{X_N}(u_2),...)_{N\in\N}$ satisfies an LDP with speed $s_N$ and rate function $J_k$. By the Dawson--G\"artner Theorem (Proposition \ref{prop: dawson-g\"artner}), the sequence $(f_{X_N}(u_1),f_{X_N}(u_2),...)_{N\in\N}$ satisfies an LDP with speed $s_N$ and rate function $I^\prime:\mathbb{Y}\to[0,\infty], I^\prime(x) = \sup_k J_k(p_k(x))$. \\
			
			Now, note that the restriction on a dense subset uniquely determines the function. More precisely, if $g,f:\mathbb{X}\to\mathbb{Y}$ are continuous and $f\big|_{\mathbb{U}} = g\big|_{\mathbb{U}}$ then $f=g$. To see this, assume that there is $x\in\mathbb{X}$ such that $f(x)\not=g(x)$. Since $\mathbb{Y}$ is Hausdorff and the functions are continuous, there are disjoint open sets $U_1,U_2\subseteq\mathbb{Y}$ (with $f(x)\in U_1$ and $g(x)\in U_2$) such that $f^{-1}(U_1)\cap g^{-1}(U_2)\subseteq \mathbb{X}$ is open and non-empty. Since $\mathbb{U}$ is dense, there is $y\in\mathbb{U}\cap f^{-1}(U_1)\cap g^{-1}(U_2)$ and by assumption $f(y) = g(y)$. But at the same time $f(y)\in U_1$ and $g(y)\in U_2$ which is a contradiction because these sets are disjoint. That means $f=g$.
			\noindent Consequently, the map 
			$$
			(f_{i}(u_1),f_{i}(u_2),...) \mapsto f_{i} \in C(\mathbb{X},\mathbb{Y}), \quad i\in\mathbb{T},
			$$ 
			is a well defined continuous injection. Here, the space $C(\mathbb{X},\mathbb{Y})$ is equipped with the topology of pointwise convergence. This topology is not metrizable but Hausdorff (as a product of Hausdorff spaces). The contraction principle (Proposition \ref{prop: contraction principle}) yields an LDP for $(f_{X_N})_{N\in\N}$ on the space $C(\mathbb{X},\mathbb{Y})$ with speed $N$ and rate function 
			$$
			I: C(\mathbb{X},\mathbb{Y}) \to[0,\infty],\quad I(f) = I^\prime(f(u_1),f(u_2),...).
			$$ 
			Since the topology of uniform convergence is finer than the topology of pointwise convergence (which is Hausdorff), \cite[Corollary 4.2.6]{dembo2009techniques} leads to the weak LDP on $(C(\mathbb{X},\mathbb{Y}), \|\,\cdot\,\|_\infty)$ which is precisely the desired statement.
		\end{proof}
		
		\noindent The next result gives a similar statement in the context of weak convergence.
		
		\begin{lemma}\label{lem: uplift topology in weak conv}
			Assume that $\mathbb{X}\subseteq C(K)$, where $K\subseteq \R^m$ is compact. Let $X_1,X_2,\ldots$ be random elements with values in $\mathbb{X}$. Suppose $\rho$ is a metric on $\mathbb{X}$ that makes $(X_N)_{N\in\N}$ tight.  If for any collection $\{u_1,\ldots u_k\}\subseteq K$, $k\in\N$, it holds that 
			$$
			(X_N(u_1),\ldots,X_N(u_k)) \xRightarrow[N\to\infty]{}(X(u_1),\ldots,X(u_k))
			$$
			for some random element $X$ in $(C(K),\rho)$. Then $X_N \xrightarrow[N\to\infty]{\dint} X$ with respect to $\rho$.
		\end{lemma}
		\begin{proof}
			A (necessary and) sufficient condition for $X_N \xRightarrow[N\to\infty]{} X$ is that each subsequence $X_{N^\prime}$ contains a further subsequence $X_{N^{\prime\prime}}$ with $X_{N^{\prime\prime}} \xRightarrow[N\to\infty]{} X$. Indeed, assume that there is a bounded continuous function $f$ and a subsequence $X_{N^\prime}$ such that 
			$$
			\Big|\int_{\mathbb{X}}f(x)\P_{N^\prime}(\dint x) - \int_{\mathbb{X}}f(x)\P (\dint x)\Big| > \epsilon
			$$
			for some $\epsilon>0$ and all $N^\prime$ where $\P_{N^\prime}$ and $\P$ are the push-forward measures for $X_{N^\prime}$ and $X$, respectively. Then, this subsequence would contain no further converging subsequence. \\
			
			Now, take any subsequence $X_{N^\prime}$. By assumption, this sequence is tight with respect to $\rho$. By Prokhorov's theorem, there is a further subsequence $X_{N^{\prime\prime}}$ and some element $Y$ in the closure of $\mathbb{X}$ such that $X_{N^{\prime\prime}}\xRightarrow[N\to\infty]{} Y$. This implies 
			$$
			(X_{N^{\prime\prime}}(u_1),\ldots,X_{N^{\prime\prime}}(u_k)) \xRightarrow[N\to\infty]{} (Y(u_1),\ldots,Y(u_k))
			$$
			by the continuous mapping theorem. It is left to show that $X=Y$. By taking a subsequence in the assumption, we have that for any collection $u_1,\ldots,u_k \in K$, 
			$$
			(X_{N^{\prime\prime}}(u_1),\ldots,X_{N^{\prime\prime}}(u_k)) \xRightarrow[N\to\infty]{} (X(u_1),\ldots,X(u_k)).
			$$
			By uniqueness of the limit, we get $Y(u) = X(u)$ for all $u\in K$ and hence $X=Y$.
		\end{proof}
		
		\subsection{Volume and support function representation}\label{sec: Volume and support function representation}
		
		The following lemma shows that the support function of a randomly projected $\ell_p^N$-ball involves a sum over random functions depending on columns of a random element in the Stiefel manifold.
		\begin{lemma}\label{lem: support function representation}
			Let $p\in (1,\infty]$, $q$ be its Hölder conjugate and $m\le N\in\N$. For an element $V_N= (v_1,\ldots,v_N)\in\mathbb{V}_{m,N}$, the support function of $N^{\frac{1}{2}-\frac{1}{q}}V_N\mathbb{B}_p^N$ satisfies the identity
			\begin{align}\label{eq: rearanged support equality}
				h(N^{\frac{1}{2}-\frac{1}{q}} V_N \mathbb{B}^N_p, u)^q = \frac{1}{N}\sum_{i=1}^N |\langle \sqrt{N}v_i, u\rangle|^q, \quad u\in\Sm. 
			\end{align}
		\end{lemma}
		\begin{proof}
			Note that $\frac{1}{2}-\frac{1}{q}=\frac{1}{p}-\frac{1}{2}$. We have			
			$$
			h(N^{\frac{1}{p}-\frac{1}{2}} V_N \mathbb{B}^N_p, u) = \sup_{x\in N^{\frac{1}{p}-\frac{1}{2}}V_N \mathbb{B}^N_p}\langle x,u\rangle = N^{\frac{1}{p}-\frac{1}{2}} \sup_{x\in \mathbb{B}^N_p}\langle V_Nx,u\rangle  = N^{\frac{1}{p}-\frac{1}{2}} \sup_{x\in \mathbb{B}^N_p}\langle x,V_N^*u\rangle  = N^{\frac{1}{p}-\frac{1}{2}}  \|V_N^* u\|_q.
			$$
	This can be written as 
			$$
			N^{\frac{1}{p}-\frac{1}{2}} \|V_N^* u\|_q  = N^{\frac{1}{2}-\frac{1}{q}}\Big( \sum_{i=1}^N |\langle v_i,u \rangle|^q \Big)^{\frac{1}{q}} = \Big(\frac{1}{N} \sum_{i=1}^N |\langle \sqrt{N}v_i,u \rangle|^q \Big)^{\frac{1}{q}}.
			$$
			A rearrangement of the terms proves the claim.
		\end{proof}
		
		\noindent We also show a representation for the volume of a random section of an $\ell_p$-ball. We see that this involves a similar sum in the integrand as in the support function of the projection.
		
		\begin{lemma}\label{lem: section formula}
			Fix $N\in\N$, $m\in\{1,\ldots,N\}$ and $p\in (1,\infty]$. Further, let $V_N=(v_1,\dots,v_N)$ be a random element distributed according to the Haar probability measure on $\mathbb{V}_{m,N}$ and $E$ be a random subspace distributed according to $\mu_{m,N}$. Then,
			\begin{align*}
				\vol_m\Big(N^{\frac{1}{p}-\frac{1}{2}}(\mathbb{B}_p^N \cap E)\Big) \overset{d}{=}  \kappa_m \int_{\mathbb{S}^{m-1}} \Big(\frac{1}{N}\sum_{i=1}^N |\langle \sqrt{N}v_i, u\rangle|^p\Big)^{-\frac{m}{p}} \,\sigma(\dint u).
			\end{align*}
		\end{lemma}
		\begin{proof} Let $\widehat{E} =  \operatorname{Range}(V_N^*)$ be the range of the linear operator $V_N^*$. Then we note that due to the uniqueness of the Haar measure on $ \mathbb{G}_{m,N}$, $\widehat{E}\in \mathbb{G}_{m,N}$ has distribution $\mu_{m,N}$, i.e.,
			$$
			\vol_m\big(N^{\frac{1}{p}-\frac{1}{2}}(\mathbb{B}_p^N \cap E)\big) \overset{d}{=} \vol_m\big(N^{\frac{1}{p}-\frac{1}{2}}(\mathbb{B}_p^N \cap \widehat{E})\big).
			$$
			For $F\in \mathbb{G}_{m,N}$ let us write $\mathbb{S}_{F}^{m-1}$ and $\sigma_{F}$ for the unit sphere and the normalized spherical Lebesgue measure in $F$. For the distance of projections of $\ell_p^N$-balls to one-dimensional subspaces we have the identity
			\begin{align}\label{eq: 1 dim projection identity}
				\rho\big(\mathbb{B}_p^N \cap F, \theta\big) = \rho\big(\mathbb{B}_p^N \cap \operatorname{span}(\theta), \theta\big) = \Big\|\frac{\theta}{\|\theta\|_p}\Big\|_2 = \frac{1}{\|\theta\|_p}
			\end{align} 
			for $\theta\in\mathbb{S}_{F}^{m-1}$. Using \eqref{eq:volume integration formula radial function} in the subspace $\widehat{E}$, we  obtain
			\begin{align*}
				\vol_m\big(N^{\frac{1}{p}-\frac{1}{2}}(\mathbb{B}_p^N \cap \widehat{E}) \big) &= \kappa_m \int_{\mathbb{S}_{\widehat{E}}^{m-1}}N^{\frac{m}{p}-\frac{m}{2}}\rho\big(\mathbb{B}_p^N \cap \widehat{E},u\big)^m\,\sigma_{\widehat{E}}(\dint u).
			\end{align*}
			Let $A$ be distributed according to the Haar probability measure on $O(N)$. We can view $V_N$ as the first $m$ columns of $A$ (which is also distributed according to Haar measure). Using this construction, we see that $A$ maps the standard basis of $\R^m$, embedded into $\R^N$, into the rows of $V_N$ and hence $A(\mathbb{S}^{m-1}) = \mathbb{S}_{\widehat{E}}^{m-1}$. By first applying the transformation $A$ to the integral and then using \eqref{eq: 1 dim projection identity}, we get
			\begin{align*}
				& \kappa_m \int_{\mathbb{S}_{\widehat{E}}^{m-1}}N^{\frac{m}{2}-\frac{m}{q}}\rho\big(\mathbb{B}_p^N \cap  \widehat{E},u\big)^m\,\sigma_{\widehat{E}}(\dint u) = \kappa_m \int_{\mathbb{S}^{m-1}}N^{\frac{m}{p}-\frac{m}{2}}\rho\Big(\mathbb{B}_p^N \cap \widehat{E}, Au \Big)^m\,\sigma(\dint u)\\
				&= \kappa_m \int_{\mathbb{S}^{m-1}}N^{\frac{m}{p}-\frac{m}{2}}\Big\|V_N^* u\Big\|_p^{-m}\,\sigma(\dint u)= \kappa_m \int_{\mathbb{S}^{m-1}} \Big(\frac{1}{N}\sum_{i=1}^N |\langle \sqrt{N}v_i, u\rangle|^p\Big)^{-\frac{m}{p}} \,\sigma(\dint u).
			\end{align*}
			This completes the argument.
		\end{proof}

		\subsection{Calculation of Hadamard differentials}\label{sec: Calculation of Hadamard differentials}
		
		\noindent We calculate the Hadamard derivative of $f\mapsto f^\alpha$. This will be needed to transform the limit theorems for the random sums of non i.i.d.~elements to the volume of the random section and projection.
		
		\begin{lemma}\label{lem: hadamard dir exp alpha}
			The map $\Phi: (C(K), \|\,\cdot\,\|_\infty) \to (C(K), \|\,\cdot\,\|_\infty)$, $\Phi(f) = f^\alpha$ for some $\alpha\in \R$ and some compact set $K\subseteq \R^m$ is Hadamard differentiable at any point $f$ such that $f(x) \not= 0$ for all $x$ with derivative $\dint\Phi(f)[h] = \alpha f^{\alpha-1}h$.
		\end{lemma}
		\begin{proof}
			We need to show that
			$$
			\frac{1}{t_N}(\Phi(f+ t_N h_N) - \Phi(f)) \xrightarrow{N\to\infty} \dint\Phi(f)[h]
			$$
			for $t_N\to 0$ and for any $h,h_1,h_2,\ldots\in C(K)$ such that $h_N \to h$ in $\|\,\cdot\,\|_{\infty}$. Using $(1+ x)^\alpha = 1 + \alpha x + o(x^2)$ for $x\to 0$ {and the fact that $f$ is bounded on $K$,} we get 
			\begin{align*}
				&\frac{1}{t_N}\Big(\Phi(f+ t_N h_N) - \Phi(f)\Big) = \frac{1}{t_N}\Big(f^\alpha\Big(1+ t_N\frac{h_N}{f}\Big)^\alpha - f^\alpha\Big)\\
				&= \frac{1}{t_N}\Big(f^\alpha\Big(1+ \alpha t_N\frac{h_N}{f} + o(t_N^2h_N^2) -1\Big)\Big) = \alpha f^{\alpha-1} h_N + f^\alpha o(t_Nh_N^2) \xrightarrow{N\to\infty} \alpha f^{\alpha-1}h = \dint\Phi(f)[h].
			\end{align*}  
			Here $f_N = o(h_N)$ means that $\frac{f_N}{h_N} \to 0$ with respect to $\|\,\cdot\,\|_\infty$.
		\end{proof}
		
		\noindent Let $C\subseteq \R^m$ be a convex body. By \cite[Page 54 and Equation (1.52)]{schneider2014convex}, we have 
		$$
		\rho(C,x) = \inf_{\langle u,x \rangle > 0, u\in \Sm} \frac{h(C,u)}{\langle x,u\rangle}.
		$$
		Remember that $\rho(C,\,\cdot\,)$ is the radial function of $C$. The map $\Phi: (C(\Sm), \|\,\cdot\,\|_\infty) \to (C(\Sm), \|\,\cdot\,\|_\infty)$ with 
		\begin{align}\label{eq: def of Phi}
			\Phi(f) = \Big(x\mapsto \inf_{\langle u,x \rangle > 0, u\in \Sm} \frac{f(u)}{\langle x,u\rangle}\Big)
		\end{align}
		is Hadamard differentiable in any constant $\mu>0$ as the next lemma shows. Together with Proposition \ref{prop: delta method for LDPs} and Proposition \ref{prop: delta method infinite dim} we see that a limit theorem for the support function holds in the same way for the radial function.
		
		\begin{lemma}\label{lem: had diff for sup to rad}
			For $M\in(0,\infty)$ let $f:\Sm\to (0,\infty)$, $x\mapsto M$ be a positive constant function. The map $\Phi$ defined by \eqref{eq: def of Phi} is Hadamard differentiable at $f$ with derivative $\dint \Phi(f)[h] = h$. 
		\end{lemma}
		\begin{proof} Going with the definition, let $t_N \to 0$ and $h_N \to h$ in the uniform norm on $C(\Sm)$. Notice that for any $x\in\Sm$, using the definition in \eqref{eq: def of Phi}, we have
			\begin{align}\label{eq:hadamard diff}
				\frac{\Phi(f + t_Nh_N)(x) - \Phi(f)(x)}{t_N} & = \frac{1}{t_N}\Bigg(\inf_{\langle u,x \rangle > 0, u\in \Sm} \frac{M + t_Nh_N(u)}{\langle x,u\rangle} -  \inf_{\langle u,x \rangle > 0, u\in \Sm} \frac{M}{\langle x,u\rangle}\Bigg)\cr
				&= \frac{1}{t_N}\Bigg(\inf_{\langle u,x \rangle > 0, u\in \Sm} \frac{M + t_Nh_N(u)}{\langle x,u\rangle} -  M\Bigg)\cr
				&= \frac{1}{t_N}\Bigg( \inf_{\langle u,x \rangle > 0, u\in \Sm} \frac{M + t_Nh_N(u) - \langle x,u \rangle M}{\langle x,u\rangle} \Bigg) \cr
				& = \inf_{\langle u,x \rangle > 0, u\in \Sm} M \frac{1-\langle x,u\rangle}{t_N\langle x,u\rangle} + \frac{h_N(u)}{\langle x,u\rangle}\cr
				&= M \frac{1-\langle x,u_N \rangle}{t_N\langle x,u_N \rangle} + \frac{h_N(u_N)}{\langle x,u_N\rangle},
			\end{align}
			where $u_N:=u_N(x) \in \Sm$ is the sequence of minimizers, which exists since the infimum is attained at some $u_N \not= 0$. Further, it holds that $u_N \to x$, since otherwise $\frac{1-\langle x,u_N \rangle}{t_N\langle x,u_N \rangle}\to\infty$ as $N\to\infty$. 
			
			We want to show that $1-\langle x,u_N \rangle = o(t_N)$ uniformly over all $x\in\Sm$. Assume now that for some $x\in\Sm$, $\langle x,u_N \rangle \le 1 - ct_N$ for some $c\in (0,1)$ for $N\in\N$ big enough along a subsequence. Then for $N$ being in the subsequence,
			\begin{align*}
				M \frac{1-\langle x,u_N \rangle}{t_N\langle x,u_N \rangle} + \frac{h_N(u_N)}{\langle x,u_N\rangle}\ge M \frac{c}{1-ct_N} + \frac{h_N(u_N)}{1-ct_N}.
			\end{align*}
			It is easy to see that $M \frac{c}{1-ct_N} + \frac{h_N(u_N)}{1-ct_N} > h_N(x)$ for $N$ large enough in that subsequence since $M\cdot c\in(0,\infty)$, which is the value attained  at $u=x$ by the function over which the infimum in \eqref{eq:hadamard diff} is taken. Hence, this is a contradiction to $u_N$ being the minimizer. We have thus shown that indeed $1-\langle x,u_N \rangle = o(t_N)$ uniformly over all $x\in\Sm$. 

In particular, this implies that for all $x\in\Sm$, we have $\|u_N-x\|_2^2=2(1-\langle x, u_N\rangle )=o(t_N)$. Thus, since 
		\[
   	           \sup_{x\in\Sm} |h_N(u_N)-h(x)| \leq \sup_{x\in\Sm}|h_N(u_N)-h_N(x)| + \|h_N-h\|_{\infty},
		\]
combining that $1-\langle x,u_N \rangle = o(t_N)$ uniformly over all $x\in\Sm$ with the uniform continuity of $h_N$ and the uniform convergence of $h_N$ to $h$, we obtain 
			\begin{align*}
				&\frac{1}{t_N}\Big(\Phi(f + t_Nh_N) - \Phi(f) \Big) = M \frac{o(t_N)}{t_N(1-o(t_N))} + \frac{h_N(u_N)}{1-o(t_N)} = M \frac{o(1)}{1-o(t_N)} + \frac{h_N(u_N)}{1 - o(t_N)} \xrightarrow{N \to \infty} h
			\end{align*} 
in the uniform norm $\|\,\cdot\,\|_{\infty}$. This completes the proof. 
		\end{proof}
		
		\subsection{Computations of moments and other preliminary results}\label{sec: Computations of moments}
		
		The next lemma gives a moment computation that will be used frequently throughout the rest of this paper and a justification to change an $N$-dependent normalization to a constant one.
		\begin{lemma}\label{lem: asym expectation}
			Let $q>-1$, $v_i$ be a column of $V_N\sim \operatorname{Unif}(\mathbb{V}_{m,N})$ for fixed $m\le N$ and $g\sim \mathcal{N}(0,1)$. Then
			\begin{align*}
				\E[|g|^q] = \frac{2^{\frac{q}{2}}\Gamma(\frac{q+1}{2})}{\sqrt{\pi}}
			\end{align*}
			and 
			$$
			\sup_{u\in\Sm} \sqrt{N}|\E[|g|^q] -\E[|\langle \sqrt{N}v_i, u\rangle|^q])| \to 0
			$$
			as $N\to\infty$
		\end{lemma}
		\begin{proof}
			\noindent By using symmetry we get
			\begin{align*}
				\E|g|^q &= (2\pi)^{-\frac{1}{2}}\int_{\R} |x|^q e^{-\frac{|x|^2}{2}} \,\dint x = 2 (2\pi)^{-\frac{1}{2}} \int_0^\infty r^{q} e^{-\frac{r^2}{2}} \,\dint r= \frac{2^{\frac{q}{2}}}{\sqrt{\pi}}  \Gamma\Big(\frac{1+q}{2}\Big).
			\end{align*}
			Consider $\E[|\langle v_i, u\rangle|^q)]$. By rotational invariance of the Haar measure on $\mathbb{V}_{m,N}$, we have $\E[|\langle v_i, u\rangle|^q)]$ = $\E[|v_{i,1}|^q)]$, where $v_{i,1}$ is the first entry of $v_i$. So the statement does not depend on $u$ or $m$. The random variable has the same distribution as the first coordinate $U_1$ from a vector $U$ uniformly distributed on $\mathbb{S}^{N-1}$. It is well known that $U_1^2$ is beta distributed with parameters $\frac{1}{2}$ and $\frac{N-1}{2}$. Using the moments of the beta distribution and its relation to the gamma function,
			$$
			\E[|v_{i,1}|^q)] = \E[(U_1^2)^{\frac{q}{2}}]  = \frac{\operatorname{Beta}(\frac{q+1}{2}, \frac{N-1}{2})}{\operatorname{Beta}(\frac{1}{2}, \frac{N-1}{2})} = \frac{\Gamma(\frac{q+1}{2})\Gamma(\frac{N-1}{2})\Gamma(\frac{N}{2})}{\Gamma(\frac{N+q}{2})\Gamma(\frac{1}{2})\Gamma(\frac{N-1}{2})} = \frac{\Gamma(\frac{q+1}{2})\Gamma(\frac{N}{2})}{\Gamma(\frac{N+q}{2})\sqrt{\pi}}.
			$$
			By Stirling's formula, we get 
			\begin{align*}
				&\sqrt{N}\Big(\frac{N^{\frac{q}{2}}\Gamma(\frac{N}{2})}{\Gamma(\frac{N+q}{2})} - 2^{\frac{q}{2}}\Big)= \frac{\sqrt{N}N^{\frac{q}{2}}\sqrt{\frac{4\pi}{N}}\big(\frac{N}{2e}\big)^\frac{N}{2}(1+O(\frac{1}{N}))}{\sqrt{\frac{4\pi}{N+q}}\big(\frac{N+q}{2e}\big)^\frac{N+q}{2}(1+O(\frac{1}{N+q}))}- \sqrt{N}2^{\frac{q}{2}}\\
				&= \frac{1+O(\frac{1}{N}))}{(1+O(\frac{1}{N+q}))}\Big(1+\frac{q/2}{N/2}\Big)^{-N/2}(2e)^{\frac{q}{2}}\sqrt{N+q}-\sqrt{N}2^{\frac{q}{2}}\\
				&\sim e^{-\frac{q}{2}}(2e)^{\frac{q}{2}}\sqrt{N+q} - \sqrt{N}2^{\frac{q}{2}} = 2^{\frac{q}{2}}(\sqrt{N+q}-\sqrt{N}) \to 0,
			\end{align*}
			which yields the claim.
		\end{proof}
		
		The following computations will be used in the proof of the central limit theorem and, more precisely, for the calculation of the variance. 
		
		\begin{lemma}\label{lem: exact_expectations}
			For $m\in\N$, let $g=(g_1,\ldots,g_m)\sim \mathcal{N}(0,\operatorname{Id}_m)$, $u=(u_1,\ldots,u_m),v=(v_1,\ldots,v_m)\in\Sm$ and $q>-1$. Then
			\begin{enumerate}
				\large
				\item $
				\E[|\langle g,u\rangle|^{q-2} \langle g,u\rangle g_i] = u_i \frac{2^{q/2}\Gamma(\frac{q+1}{2})}{\sqrt{\pi}} = u_i\E[|g_1|^q],
				$
				\item $
				\E[|\langle g,u\rangle|^q g_i^2] = \frac{2^{q/2}\Gamma(\frac{q+1}{2})(1+qu_i^2)}{\sqrt{\pi}} = (1+qu_i^2)\E[|g_1|^q]
				$
				\item For $i\not=j$,
				$
				\E[|\langle g,u\rangle|^q g_ig_j] = u_iu_j\frac{q2^{q/2}\Gamma(\frac{q+1}{2})}{\sqrt{\pi}} = qu_iu_j\E[|g_1|^q]
				$
				\item For $u\not=v$,
				$$
				\E[|\langle g,u\rangle\langle g,v\rangle|^q] = \frac{2^q\Gamma(\frac{q+1}{2})^2}{\pi} {}_2 F_1(-q/2,-q/2, 1/2, \langle u,v\rangle^2)
				$$
				where ${}_2F_1$ denotes the Gauss hypergeometric function.
				\item For $u=v$,
				$
				\E[|\langle g,u\rangle\langle g,u\rangle|^q] = \frac{2^q \Gamma(\frac{1}{2}+q)}{\sqrt{\pi}}.
				$
			\end{enumerate}
		\end{lemma}
		
		\begin{proof}
			The proof is based on the fact that
			$$
			(\langle g,u\rangle, \langle g,v\rangle) = \sum_{i=1}^m g_i(u_i,v_i) \sim \mathcal{N}\Big(0,\begin{pmatrix}
				1 & \langle u,v\rangle \\
				\langle u,v\rangle & 1
			\end{pmatrix}\Big)
			$$
			holds for any $u,v\in\Sm$. It is well known that for any joint Gaussian distribution also a conditional Gaussian distribution between the components holds. In this case,
			$$
			\langle g,u\rangle \big| \langle g,v \rangle \sim \mathcal{N}(\langle u,v\rangle\langle g,v \rangle, 1-\langle u,v\rangle^2).
			$$
			We combine this with the law of total probability and the fact that 
			$$
			\E[|g_1|^q] = \frac{2^{q/2}\Gamma(\frac{q+1}{2})}{\sqrt{\pi}}
			$$
			in each case (see Lemma \ref{lem: asym expectation}). For the first one just choose $v$ as the vector that has only a one at the $i$-th entry. This leads to
			\begin{align*}
				\E[|\langle g,u\rangle|^{q-2} \langle g,u\rangle g_i] = \E[|\langle g,u\rangle|^{q-2} \langle g,u\rangle\E[g_i|\langle g,u\rangle]] = u_i  \E[|\langle g,u\rangle|^{q}] = u_i \frac{2^{q/2}\Gamma(\frac{q+1}{2})}{\sqrt{\pi}},
			\end{align*}
			which is what we wanted to show. For the second one,
			\begin{align*}
				\E[|\langle g,u\rangle|^q g_i^2] = \E[|\langle g,u\rangle|^{q}\E[g_i^2| \langle g,u\rangle]],
			\end{align*}
			and again using the conditional expectation 
			$$
			\E[g_i^2| \langle g,u\rangle] = \Var(g_i| \langle g,u\rangle) + \E[g_i| \langle g,u\rangle]^2 = 1-u_i^2+ u_i^2\langle g,u\rangle^2,
			$$
			which leaves us with
			\begin{align*}
				\E[|\langle g,u\rangle|^{q}\E[g_i^2| \langle g,u\rangle]] &= (1-u_i^2)\E[|\langle g,u\rangle|^q] + u_i\E[|\langle g,u\rangle|^{q+2}]\\
				& = (1-u_i^2)\frac{2^{q/2}\Gamma(\frac{q+1}{2})}{\sqrt{\pi}} + u_i\frac{2^{q/2+1}\Gamma(\frac{q+3}{2})}{\sqrt{\pi}} .
			\end{align*}
			This simplifies to the desired expression using the property $\Gamma(x+1) = x\Gamma(x)$ for any $x>0$. \\
			We continue similarly for the third one. This time we will need that similar as above
			$$
			(g_i,g_j, \langle g,u \rangle) \sim \mathcal{N}\Bigg(0,\begin{pmatrix}
				1 & 0 & u_i \\
				0 & 1 & u_j &\\
				u_i & u_j & 1
			\end{pmatrix}\Bigg),
			$$
			which gives  
			$$
			(g_i,g_j)\big| \langle g,u\rangle \sim \mathcal{N}\Bigg(\langle g,u\rangle (u_i,u_j), \begin{pmatrix}
				1 - u_i^2 & -u_iu_j\\
				-u_iu_j & 1-u_j^2
			\end{pmatrix}\Bigg).
			$$
			Now continue with 
			\begin{align*}
				\E[|\langle g,u\rangle|^qg_ig_j] = \E[|\langle g,u\rangle|^q\E[g_ig_j| \langle g,u\rangle]]
			\end{align*}
			and using the conditional distribution from above
			\begin{align*}
				\E[g_ig_j| \langle g,u\rangle] &= \Cov[g_i,g_j| \langle g,u\rangle] + \E[g_i|\langle g,u\rangle]\E[g_j|\langle g,u\rangle]\\
				&= -u_iu_j + u_i\langle g,u\rangle u_j\langle g,u\rangle = u_iu_j(\langle g,u\rangle^2-1),
			\end{align*}
			which gives the final result
			\begin{align*}
				\E[|\langle g,u\rangle|^q\E[g_ig_j| \langle g,u\rangle]] &= u_iu_j \E[|\langle g,u\rangle|^q(\langle g,u\rangle^2-1))] = u_iu_j(\E[|g_1|^{q+2}] - \E[|g_1|^q])\\
				& = \frac{u_iu_j}{\sqrt{\pi}}(2^{q/2+1}\Gamma(\frac{q+3}{2})- 2^{q/2}\Gamma(\frac{q+1}{2}))= u_iu_j\frac{q2^{q/2}\Gamma(\frac{q+1}{2})}{\sqrt{\pi}},
			\end{align*}
			using again $\Gamma(x+1) = x\Gamma(x)$ in the last step. We are left with the fourth expectation. The case $u=v$ follows simply from the rotational invariance:
			$$
			\E[|\langle g,u\rangle|^{2q}] = \E[|g_1|^{2q}] = \frac{2^q \Gamma(\frac{1}{2}+q)}{\sqrt{\pi}}.
			$$
			The case $u\not= v$ is a bit more involved. We start as in the other cases with 
			$$
			\E[|\langle g,u\rangle\langle g,v\rangle|^q] = \E[|\langle g,v\rangle|^q \E[|\langle g,u\rangle|^q |\langle g,v\rangle]].
			$$
			Calculating the inner expectation $\E[|\langle g,u\rangle|^q |\langle g,v\rangle]$ involves the $q$-th absolute moment of a non-centered Gaussian random variable (with mean $\langle u,v\rangle\langle v,g\rangle$ and variance $1-\langle u,v\rangle^2$). By \cite{winkelbauer2012moments}, we have 
			$$
			\E\Big[|\langle g,u\rangle|^q \Big|\langle g,v\rangle\Big] = (1-\langle u,v\rangle^2)^{q/2}2^{q/2}\frac{\Gamma(\frac{q+1}{2})}{\sqrt{\pi}} {}_1F_1\Big(-q/2, 1/2,-\frac{\langle u,v\rangle^2 \langle g,v\rangle^2}{2(1-\langle u,v\rangle^2)}\Big),
			$$
			where ${}_1F_1$ denotes Kummer's hypergeometric function (see \cite[Chapter 15]{lozier2003nist}). By definition of this function we have
			\begin{align*}
				{}_1F_1\Big(-q/2, 1/2,-\frac{\langle u,v\rangle^2 \langle g,v\rangle^2}{2(1-\langle u,v\rangle^2)}\Big) = \sum_{n=0}^\infty \frac{\Gamma(-q/2+n)}{\Gamma(-q/2)}\frac{\sqrt{\pi}}{\Gamma(n+1/2)}\frac{(-1)^n}{n!}\big(\frac{\langle u,v\rangle^2}{2(1-\langle u,v\rangle^2)} \big)^n \langle g,v\rangle^{2n}.
			\end{align*}
			This implies
			\begin{align*}
				\E[|\langle g,u\rangle|^q |\langle g,v\rangle] =  (1-\langle u,v\rangle^2)^{q/2}2^{q/2}\frac{\Gamma(\frac{q+1}{2})}{\Gamma(-q/2)}\sum_{n=0}^\infty \frac{\Gamma(-q/2+n)}{\Gamma(n+1/2)}\frac{(-1)^n}{n!}\Big(\frac{\langle u,v\rangle^2}{2(1-\langle u,v\rangle^2)} \Big)^n \E[|g_1|^{q+2n}]
			\end{align*}
			since this sum converges absolutely. Now we can use the central moment to conclude that the above is the same as 
			$$
			(1-\langle u,v\rangle^2)^{q/2}2^{q/2}\frac{\Gamma(\frac{q+1}{2})}{\sqrt{\pi}\Gamma(-q/2)}\sum_{n=0}^\infty \frac{\Gamma(-q/2+n)}{\Gamma(n+1/2)n!}\Big(-\frac{\langle u,v\rangle^2}{2(1-\langle u,v\rangle^2)} \Big)^n 2^{n+q/2}\Gamma\big(\frac{q+2n+1}{2}\big).
			$$
			By writing out the definition of the Gauss hypergeometric function we see
			\begin{align*}
				&\sum_{n=0}^\infty \frac{\Gamma(-q/2+n)}{\Gamma(n+1/2)n!}\big(-\frac{\langle u,v\rangle^2}{2(1-\langle u,v\rangle^2)} \big)^n 2^{n+q/2}\Gamma\big(\frac{q+2n+1}{2}\big)\\
				&= \frac{2^{q/2}\Gamma(-q/2)\Gamma(\frac{q+1}{2})}{\sqrt{\pi}}{}_2F_1\Big(-q/2,\frac{q+1}{2},1/2,-\frac{\langle u,v\rangle^2}{1-\langle u,v\rangle^2}\Big),
			\end{align*}
			which yields 
			$$
			\E[|\langle g,u\rangle\langle g,v\rangle|^q] = \frac{(1-\langle u,v\rangle^2)^{q/2}2^q\Gamma(\frac{q+1}{2})^2 {}_2 F_1(-q/2, \frac{q+1}{2},\frac{1}{2},-\frac{\langle u,v\rangle^2}{1-\langle u,v\rangle^2})}{\pi}.
			$$
			After a Pfaff transformation (see \cite[Equation 15.8.1]{lozier2003nist}), we end up with
			$$
			\frac{(1-\langle u,v\rangle^2)^{q/2}2^q\Gamma(\frac{q+1}{2})^2 {}_2 F_1(-q/2, \frac{q+1}{2},\frac{1}{2},-\frac{\langle u,v\rangle^2}{1-\langle u,v\rangle^2})}{\pi} = \frac{2^q\Gamma(\frac{q+1}{2})^2}{\pi} {}_2 F_1(-q/2,-q/2, 1/2, \langle u,v\rangle^2),
			$$
			as desired.
		\end{proof}
		
		The expectation in the next lemma will be needed later in the calculation of the variance for the volume. 
		
		\begin{lemma}\label{lem: expectation gauss part with int}
			Let $p,q \ge 1$, $m\in\mathbb{N}$, and $g= (g_1,\ldots,g_m) \sim \mathcal{N}(0,\operatorname{Id}_m)$. Then, 
			$$
			\E\Big[ \int_{\mathbb{S}^{m-1}}\int_{\mathbb{S}^{m-1}}|\langle g,u\rangle|^p|\langle g,v\rangle|^q\,\sigma(\dint u)\sigma(\dint v)\Big]=\frac{2^{\frac{p+q}{2}+1}\Gamma(\frac{m+p + q}{2})\Gamma(\frac{1+p}{2})\Gamma(\frac{1+q}{2})\Gamma(1+\frac{m}{2})}{m\pi \Gamma(\frac{m+p}{2})\Gamma(\frac{m+q}{2})}.
			$$
		\end{lemma}
		\begin{proof} By changing to spherical coordinates, we get
			\begin{align*}
				\E\big[\|g\|^p\big] &= (2\pi)^{-\frac{m}{2}}\int_{\R^m} \|x\|^p e^{-\frac{\|x\|^2}{2}} \,\dint x = m\kappa_m (2\pi)^{-\frac{m}{2}} \int_0^\infty r^{m+p-1} e^{-\frac{r^2}{2}} \,\dint r\\
				&= \frac{m\kappa_m 2^{\frac{p}{2}-1}}{\pi^{\frac{m}{2}}}  \Gamma\Big(\frac{m+p}{2}\Big) = \frac{m2^{\frac{p}{2}-1}\Gamma(\frac{m+p}{2})}{\Gamma(1+\frac{m}{2})}.
			\end{align*}
			We use \cite[Lemma 2.1]{alonso2008isotropy} and the result above to see that 
			\begin{align*}
				&\E\Big[\int_{\mathbb{S}^{m-1}}\int_{\mathbb{S}^{m-1}}|\langle g,u\rangle|^p|\langle g,v\rangle|^q\,\sigma(\dint u)\sigma(\dint v)\Big]\\ &= \E\Big[\|g\|_2^{p+q}\Big(\int_{\mathbb{S}^{m-1}} \Big|\Big\langle \frac{g}{\|g\|_2},u\Big\rangle\Big|^p\,\sigma(\dint u)\Big)\Big(\int_{\mathbb{S}^{m-1}} \Big|\Big\langle \frac{g}{\|g\|_2},u\Big\rangle\Big|^q\,\sigma(\dint u)\Big)\Big]\\ 
				&=\frac{2\Gamma(\frac{1+p}{2})\Gamma(1+\frac{m}{2})}{\sqrt{\pi}m\Gamma(\frac{m+p}{2})}\frac{2\Gamma(\frac{1+q}{2})\Gamma(1+\frac{m}{2})}{\sqrt{\pi}m\Gamma(\frac{m+q}{2})}\E\big[\|g\|_2^{p+q}\big]\\
				&=  \frac{2^{\frac{p+q}{2}+1}\Gamma(\frac{m+p + q}{2})\Gamma(\frac{1+p}{2})\Gamma(\frac{1+q}{2})\Gamma(1+\frac{m}{2})}{m\pi \Gamma(\frac{m+p}{2})\Gamma(\frac{m+q}{2})},
			\end{align*}
			which completes the proof.
		\end{proof}

		The following is an elementary Lipschitz estimate that we will need in order to apply Proposition \ref{prop: functional CLT}. 
		
		\begin{lemma}\label{lem: lipschitz condition}
			Assume $N\in \N$,  $x_1,x_2,\ldots,x_N\in \R^m$, $q\ge 1$ and $s_N\ge 0$ for all $N\in\N$. Further, let $K\subseteq \R^m$ be compact and $\sup_{u\in K}\|u\|_2 \le C$ for some $C\in(0,\infty)$. Then, we have 
			$$
			\Big| \frac{1}{N}\sum_{i=1}^N |\langle x_i,u\rangle|^q - \frac{1}{N}\sum_{i=1}^N |\langle x_i,v\rangle|^q \Big| \le \Big(\frac{1}{N}\sum_{i=1}^N qC^{q-1}\|x_i\|_2^{q-1}\Big) \|u-v\|_2
			$$ 
			all $u, v \in K$.
		\end{lemma}
		\begin{proof}
			Clearly, the composition of Lipschitz functions is again Lipschitz with the Lipschitz constant for the composition being the product of the respective Lipschitz constants of the functions. Since $u\mapsto |\langle x, u\rangle|$ has Lipschitz constant $1$, $x\mapsto x^q$ has Lipschitz constant $qa^{q-1}$ on the interval $[0,a]$, and $|\langle x,u\rangle| \le C \|x\|_2$, we find that 
			$$
			\Big| |\langle x_i,u\rangle|^q - |\langle x_i,v\rangle|^q \Big| \le qC^{q-1}\|x\|_2^{q-1} \|u-v\|_2.
			$$ 
			Consequently,
			$$
			\Big| \frac{1}{N}\sum_{i=1}^N |\langle x_i,u\rangle|^q - \frac{1}{N}\sum_{i=1}^N |\langle x_i,v\rangle|^q \Big| \le \Big(\frac{1}{N}\sum_{i=1}^N qC^{q-1}\|x_i\|_2^{q}\Big) \|u-v\|_2,
			$$
			as desired.
		\end{proof}
		
		\subsection{Proof of the central limit theorem and the moderate deviation principle}\label{sec: CLT and MDP proof}
		
		In this section we prove the functional results Theorem \ref{thm: Functional CLT} and Theorem \ref{thm: Functional MDP}. We now provide an outline for a general method of proof for Theorem \ref{thm: Functional CLT}. The sequence in the theorem is a sum of random functions depending on columns of a random element in the Stiefel manifold. Every column has the same distribution as $(G_NG_N^*)^{-\frac{1}{2}}g_i$, where $G_N\in\R^{m\times N}$ is a standard Gaussian matrix with i.d.d.\, columns $g_1,\ldots,g_N \sim \mathcal{N}(0,\operatorname{Id}_m)$. Using a Taylor expansion, we can separate $(G_NG_N^*)^{-\frac{1}{2}}$ from $g_i$ and reduce this problem to sums over i.d.d.~random functions to which we can apply known limit theorems. These can be transformed back to the original sequence which proves the desired statement. The proof of the MDP follows similar steps using the corresponding theorem in the moderate deviations regime. \\
		
		We start with a central limit theorem for case of i.i.d. random function that will be used later.
		
		\begin{lemma}\label{lem: CLT for gaussian sup version of ellp}
			Let $m\in\N$, $q\ge 0$, $K\subseteq \R^m$ be compact and $g_1,g_2,\ldots$ be i.i.d.~random vectors in $\R^m$ with $g_1\sim\mathcal{N}(0,\operatorname{Id}_m)$. Then the sequence of random functions $\Phi_{N} : K \to \R$ given by
			$$
			\Phi_{N}(u) = \frac{1}{\sqrt{N}}\sum_{i=1}^N (|\langle g_i, u\rangle|^q - \E[|\langle g_i, u\rangle|^q]),
			$$
			converges in distribution to a centered Gaussian process $Z_1$ indexed by $K$ with 
			$$
			\E[ Z_1(u_1)Z_1(u_2)] = \E[(|\langle g_i,u_1\rangle|^{q} - \E[|\langle g_i,u_1\rangle|^q])(|\langle g_i,u_2\rangle|^{q} - \E[|\langle g_i,u_2\rangle|^{q}])]
			$$
			for all $u_1,u_2\in K$.
		\end{lemma}
		\begin{proof} We check the assumptions of Proposition \ref{prop: functional CLT}. We find a constant $c>0$ such that $K\subseteq c\mathbb{B}_2^{m}$. The summands are centered and 
			$$
			\V[|\langle g_i,u\rangle|^{q}] = \E\Big[(|\langle g_i,u\rangle|^{q} - \E[|\langle g_i,u\rangle|^{q}])^2\Big]< \infty.
			$$ 
			Lemma \ref{lem: lipschitz condition} for $N=1$ gives the Lipschitz condition. Finally, the covering number with respect to a radius $\varepsilon>0$, denoted by $\operatorname{Covering}(\mathbb{B}_2^m, \|\,\cdot\,\|,\epsilon)$, is upper bounded by $(\frac{2}{\epsilon}+1)^{m}$ for any $k$ according to \cite[Corollary 4.2.13]{vershynin2018high}. Thus,
			$$
			\int_{0}^1 \log\operatorname{Covering}(c\mathbb{B}_2^m,\|\,\cdot\,\|,\epsilon)\,\dint \epsilon = m\int_0^1\log\Big(1+{2c\over \varepsilon}\Big)\,\dint\varepsilon <\infty.
			$$
			We can now apply Proposition \ref{prop: functional CLT} to conclude that the sequence of random processes $(\Phi_{N})_{N\in\N}$ converges, as $N\to\infty$, in distribution with respect to the topology of uniform convergence to a centered Gaussian process $Z_1(u)$, $u\in K$ with 
			$$
			\E[ Z_1(u_1)Z_1(u_2)] = \E[(|\langle g_i,u_1\rangle|^{q} - \E[|\langle g_i,u_1\rangle|^q])(|\langle g_i,u_2\rangle|^{q} - \E[|\langle g_i,u_2\rangle|^{q}])]
			$$
			for all $u_1,u_2$. 
		\end{proof}
		
		\noindent For $0<\delta<1$, define 
		\begin{align}\label{eq: K(m,delta) defi}
			\mathbb{K}(m,\delta) := \Big\{M\in\R^{m\times m}: 1-\delta \le \lambda_i \le 1+ \delta \text{ for all eigenvalues } \lambda_i \text{ of } M, M \text{ symmetric }\Big\}.
		\end{align}
		Since $M$ is symmetric the eigenvalues are real. Let $M\in \mathbb{K}(m,\delta)$. Then, $\det(M) = \prod_{i=1}^m \lambda_i>0$, $M$ is positive definite and any norm of $M$ is bounded. Further, since eigenvalues are continuous in the entries of the matrix, the set is closed and hence, $\mathbb{K}(m,\delta)$ is compact. We define 
		$$
		\mathbb{K}(m,\delta)\ni G_{N,\delta} := \begin{cases}
			G_N & : \quad \frac{1}{N}G_NG_N^* \in \mathbb{K}(m,\delta)\\
			\operatorname{Id}_m &: \quad \text{ otherwise } .
		\end{cases}
		$$
		
		\begin{proof}[Proof of Theorem \ref{thm: Functional CLT}] We want to find the weak limit of the sequence of processes
			\begin{align}
				\frac{1}{\sqrt{N}}\sum_{i=1}^N (|\langle \sqrt{N}v_i, \,\cdot\,\rangle|^q - \E[|\langle \sqrt{N}v_i, \,\cdot\,\rangle|^q])
			\end{align}
			in the uniform norm indexed by $\Sm$. By Lemma \ref{lem: asym expectation}, $v_i \overset{d}{=} (G_NG_N^*)^{-\frac{1}{2}}g_i$ and Slutsky's theorem, this sequence satisfies the same weak convergence as
			$$
			\frac{1}{\sqrt{N}}\sum_{i=1}^N \Big(|\langle (\frac{1}{N}G_NG_N^*)^{-\frac{1}{2}}g_i, \,\cdot\,\rangle|^q - \E[|g|^q]\Big),
			$$
			where $g\sim\mathcal{N}(0,1)$, $g_1,\ldots,g_N \sim \mathcal{N}(0,\operatorname{Id}_m)$ are all independent. It is therefore sufficient to establish the central limit theorem for the latter sequence. Let 
			$$
			S^m_+ := \Big\{M\in\R^{m\times m}\,:\, M \text{ is positive definite and symmetric} \Big\}.
			$$
			It is well known that the set $S^m_+$ is open in the set of symmetric matrices. Further, with the notation from above we have
			\begin{align*}
				&\frac{1}{\sqrt{N}}\sum_{i=1}^N (|\langle \sqrt{N}v_i, u\rangle|^q - \E[|g|^q])\overset{d}{=} \frac{1}{\sqrt{N}}\sum_{i=1}^N \Big( \Phi\Big(\Big(\frac{1}{N}G_NG_N^*\Big)^{-\frac{1}{2}}, u, g_i\Big)  - \E[|g|^q]\Big),
			\end{align*}
			where $\Phi: S^m_+\times \Sm \times \R^m \to \R$ is defined by 
			$$
			\Phi(M, u, g_i) := |\langle Mg_i, u\rangle|^q.
			$$
			To shorten notation, let $F_N = (\frac{1}{N}G_NG_N^*)^{-\frac{1}{2}}$. We show some properties about $\Phi$ and $F_N$ that we will need later on. A $m\times m$ matrix $M\in S^m_+$ has entries $(m_{i,j})_{i,j=1}^m$ and $m_{i,j} = m_{j,i}$ due to symmetry. For $q>1$, $\Phi$ is continuously differentiable in all points (since $x\mapsto |x|^q$ has continuous derivative $q|x|^{q-2}x$ in $x\not= 0$) with 
			\begin{align*}
				\frac{\dint}{\dint m_{i,j}}\Phi(M,u,g) = \begin{cases}
					q|\langle Mg,u \rangle |^{q-2}\langle Mg,u \rangle (g_j u_i+g_i u_j) &: \quad i \not= j\\
					qu_i|\langle Mg,u \rangle |^{q-2}\langle Mg,u \rangle g_i &: \quad i= j,
				\end{cases} 
			\end{align*}
			due to the symmetry of $M$. We remark that in the case $q=1$ the function is almost everywhere continuously differentiable in $M$. We will need 
			\begin{align*}
				e_{i,j}(u) &:= \E[\frac{\dint}{\dint m_{i,j}}\Phi(\operatorname{Id}_m,u,g)]\\
				&= \begin{cases}
					q u_i \E[|\langle g,u\rangle|^{q-2}\langle g,u\rangle g_j] + q u_j \E[|\langle g,u\rangle|^{q-2}\langle g,u\rangle g_i] &: \quad i \not= j\\
					qu_i\E[|\langle Mg,u \rangle |^{q-2}\langle Mg,u \rangle g_i] &:\quad i= j.
				\end{cases} 
			\end{align*}
			Further, we know by the strong law of large numbers and the continuous mapping theorem that $F_N \xrightarrow{a.s.} \operatorname{Id}_m$. By the multivariate central limit theorem (\cite[Theorem 11.10]{breiman1992probability}), we know 
			$$
			\sqrt{N}\Big(\frac{1}{N}G_NG_N^* - \operatorname{Id}_m\Big) \xrightarrow{d} Z\sim \mathcal{N}(0,\Sigma),
			$$
			where $\Sigma$ is a $m^2 \times m^2$ matrix with entries $\Sigma_{(i,j),(k,\ell)}$ given by the covariance of $g_{1,k}g_{1,\ell}$ and $g_{1,i}g_{1,j}$, i.e.
			\begin{align*}
				\E[g_{1,k}g_{1,\ell}g_{1,i}g_{1,j}]-\E[g_{1,k}g_{1,\ell}]\E[g_{1,i}g_{1,j}] = \begin{cases}
					\E[g_{1,k}^4]-\E[g_{1,k}^2]^2 = 2 &:\quad i=j=k=\ell \\
					\E[g_{1,k}^2]\E[g_{1,\ell}^2]=1   &: \quad i=k\not= j=\ell \text{ or } j=k\not= i=\ell\\
					0 &: \quad \text{otherwise. }
				\end{cases}
			\end{align*}
			Since $\Psi(M) := M^{-\frac{1}{2}}$ is a differentiable map in $\operatorname{Id}_m$ on the space of symmetric matrices, the delta method (Proposition \ref{prop: delta method infinite dim}) implies that 
			$$
			\sqrt{N}(F_N - \operatorname{Id}_m) \xrightarrow{d} -\frac{1}{2}Z,
			$$
			since by matrix calculus we know $\dint\Psi(\operatorname{Id}_m)[Z] = -\frac{1}{2}Z$ (combine, e.g., \cite[Theorem 1.1]{del2018taylor} and \cite[Theorem 8.3]{magnus2019matrix} to $\operatorname{Id}_m$ using the chain rule). Note that almost surely, $\Phi(\,\cdot\,, u,g_i)$ is differentiable in a neighborhood of $\operatorname{Id}_m$. We do a Taylor expansion of $\Phi$ in $\operatorname{Id}_m$ with integral remainder (see \cite[Theorem 2]{folland2005higher}). This means
			\begin{align}
				&\frac{1}{\sqrt{N}}\sum_{i=1}^N (|\langle \sqrt{N}v_i, u\rangle|^q - \E[|g|^q])\notag \\
				&\overset{d}{=} \frac{1}{\sqrt{N}}\sum_{i=1}^N \Phi(\operatorname{Id}_m,u,g_i) -\E[|g|^q] + \sum_{k=1}^m\sum_{\ell = k}^m (F_N - \operatorname{Id}_m)_{k,\ell}\int_0^1 \frac{1}{\sqrt{N}}\sum_{i=1}^N \frac{\dint}{\dint m_{k,\ell}}\Phi(\operatorname{Id}_m + t(F_N-\operatorname{Id}_m),u,g_i)\dint t \notag\\
				&= \frac{1}{\sqrt{N}}\sum_{i=1}^N |\langle g_i,u\rangle|^q -\E[|g|^q] + \sum_{k=1}^m\sum_{\ell = k}^m (\sqrt{N}(F_N - \operatorname{Id}_m))_{k,\ell} \frac{1}{N}\sum_{i=1}^N\int_0^1  \frac{\dint}{\dint m_{k,\ell}}\Phi(\operatorname{Id}_m + t(F_N-\operatorname{Id}_m),u,g_i)\dint t. \label{eq: CLT proof taylor exp}
			\end{align}
			Next, we show that 
			\begin{align}\label{eq: CLT proof def c1}
				Z^{k,\ell}_N(\,\cdot\,) = \frac{1}{N}\sum_{i=1}^N\int_0^1  \frac{\dint}{\dint m_{k,\ell}}\Phi(\operatorname{Id}_m + t(F_N-\operatorname{Id}_m),u,g_i)\dint t \xrightarrow{a.s.} c_1^{k,\ell}(\,\cdot\,)
			\end{align}
			for some deterministic element $c_1^{k,\ell}\in C(\Sm)$. First, consider the version of $Z_N$ that restricts $F_N$ to the matrix $F_{N,\delta} := \big(\frac{1}{N}G_{N,\delta}G_{N,\delta}^*\big)^{-\frac{1}{2}}$ which is defined on a compact set around $\operatorname{Id}_m$ in the space of symmetric positive definite matrices:
			$$
			\widehat{Z_N}^{k,\ell}(u) := \frac{1}{N}\sum_{i=1}^N\int_0^1  \frac{\dint}{\dint m_{k,\ell}}\Phi(\operatorname{Id}_m + t(F_{N,\delta}-\operatorname{Id}_m),u,g_i)\,\dint t.
			$$
			Let $\epsilon>0$ and notice $\P(\| Z^{k,\ell}_N - \widehat{Z_N}^{k,\ell}\|_\infty \ge \epsilon) \le \P(\frac{1}{N}G_NG_N^* \not\in \mathbb{K}(m,\delta)) \to 0$. This mean if $\widehat{Z_N}^{k,\ell} \xrightarrow{a.s.} c_1^{k,\ell}$, then $Z_N^{k,\ell} \xrightarrow{\P} c_1^{k,\ell}$. 
			For $M\in  \mathbb{K}(m,\delta)$, consider the pair 
			\begin{align}\label{eq: CLTproofZ_NasTupel}
				\Big((u,M) \mapsto  \frac{1}{N}\sum_{i=1}^N \int_0^1 \frac{\dint}{\dint m_{k,\ell}}\Phi(\operatorname{Id}_m + t(M-\operatorname{Id}_m),u,g_i)\dint t, \frac{1}{N}(G_{N,\delta}G_{N,\delta}^*)\Big),
			\end{align}
			which is an element in $C(\Sm \times \mathbb{K}(m,\delta))\times \mathbb{K}(m,\delta)$. This is an empirical mean over i.i.d.\, random elements in a separable Banach space. Further, we have that 
			$$
			\E\big[\| (u,M)\mapsto \int_0^1 \frac{\dint}{\dint m_{k,\ell}}\Phi(\operatorname{Id}_m + t(M-\operatorname{Id}_m),u,g_i)\dint t \|_\infty\big] \le \int_0^1 \E[q\|(\operatorname{Id}_m + tM)g\|^{q-1}g_\ell]\dint t<\infty.
			$$
			By the strong law of large number for separable Banach spaces (\cite[Corollary 7.10]{ledoux2013probability}), we have convergence to the pair 
			$$
			\Big(\E\Big[(u,M)\mapsto \int_0^1\frac{\dint}{\dint m_{k,\ell}}\Phi(\operatorname{Id}_m + t(M-\operatorname{Id}_m),u,g_i)\dint t \Big], \operatorname{Id}_m\Big)
			$$
			in the space $C(\Sm \times \mathbb{K}(m,\delta))\times \mathbb{K}(m,\delta)$ in the sense of the Bochner integral. The continuous mapping theorem applied to the function $(f,M)\mapsto f(\,\cdot\,, M^{-\frac{1}{2}})$ implies 
			$$
			\widehat{Z_N}(\,\cdot\,)^{k,\ell} \xrightarrow{a.s.}  \E\Big[\int_0^1\frac{\dint}{\dint m_{k,\ell}}\Phi(\operatorname{Id}_m + t(\operatorname{Id}_m -\operatorname{Id}_m),\,\cdot\,,g_i)\dint t \Big] = \E[ \frac{\dint}{\dint m_{k,\ell}}\Phi(\operatorname{Id}_m,\,\cdot\,,g_i) ] = e_{k,\ell}(\,\cdot\,),
			$$
			which is what we wanted to show. \\
			
			By Lemma \ref{lem: CLT for gaussian sup version of ellp}, we have $\frac{1}{\sqrt{N}}\sum_{i=1}^N |\langle g_i,\,\cdot\,\rangle|^q -\E[|g|^q] \xrightarrow{d} Z_1$ with 
			$$
			\E[Z_1(u_1)Z_1(u_2)] = \E[|\langle g_1,u_1\rangle \langle g_1,u_2\rangle|^q] - \E[|g|^q]^2.
			$$
			Further, we know from above that 
			$$
			\sqrt{N}(F_N - \operatorname{Id}_{m}) \xrightarrow{d} Z_2 \sim \mathcal{N}(0,\Sigma).
			$$
			
			By Prohorovs theorem (\cite[Theorem 23.2]{kallenberg1997foundations}) both sequences are tight which implies that the pair
			$$
			\Big(\frac{1}{\sqrt{N}}\sum_{i=1}^N |\langle g_i,(\,\cdot\,)\rangle|^q -\E[|g|^q], \sqrt{N}(F_N - \operatorname{Id}_{m})\Big)
			$$
			is also tight (even without independence). For any $\{u_1,\ldots,u_k\} \subseteq \Sm$ the multivariate central limit theorem implies that
			$$
			\Big(\frac{1}{\sqrt{N}}\sum_{i=1}^N |\langle g_i,u_1\rangle|^q-\E[|g|^q],\ldots,|\langle g_i,u_k\rangle|^q-\E[|g|^q], \frac{1}{\sqrt{N}}G_NG_N^* - \operatorname{Id}_m\Big)
			$$
			converges weakly to a centered multivariate Gaussian vector with its corresponding covariance matrix on $\R^{k}\times \R^{m\times m}$. The two individual central limit theorems together with tightness and Lemma \ref{lem: uplift topology in weak conv} imply the weak convergence of 
			$$
			\Big(\frac{1}{\sqrt{N}}\sum_{i=1}^N (|\langle g_i,\,\cdot\,\rangle|^q-\E[|g|^q], \frac{1}{\sqrt{N}}G_NG_N^* - \operatorname{Id}_m\Big) \xRightarrow[N\to\infty]{} Z = (Z_1,Z_2)
			$$
			since the limit is determined by the projections. Here the tuple $Z = (Z_1,Z_2)$ is a Gaussian process (all projections are Gaussian) indexed by $\Sm\cup \{1,\ldots,m\}^2$. The covariance is given by 
			$$
			\E[Z(s)Z(t)] = \begin{cases}
				\E[Z_1(s)Z_1(t)] = \E[|\langle g_1,u_1\rangle \langle g_1,u_2\rangle|^q] - \E[|g|^q]^2 &: \quad s,t \in \Sm \\
				\E[Z_1(s)Z_2(t)] = \E[(|\langle g_1,u_1\rangle|^q-\E[|g|^q])g_{t_1}g_{t_2}]   & : \quad s\in \Sm,    t=(t_1,t_2)\in \{1,\ldots,m\}^2\\
				\E[Z_2(s)Z_2(t)] = \Sigma_{s,t} & : \quad s,t \in \{1,\ldots,m\}^2.
			\end{cases}
			$$
			The delta method applied to $(f,M) \mapsto (f,M^{-\frac{1}{2}})$ gives the weak convergence of 
			$$
			\sqrt{N}\Big(\big(\frac{1}{N}\sum_{i=1}^N|\langle g_i,\,\cdot\,\rangle|^q\big), F_N) - (\E[|g|^q], \operatorname{Id}_m)\Big) = \Big(\frac{1}{\sqrt{N}}\sum_{i=1}^N (|\langle g_i,\,\cdot\,\rangle|^q-\E[|g|^q]), \sqrt{N}(F_N-\operatorname{Id}_m)\Big) 
			$$
			to the tuple $Z=(Z_1, -\frac{1}{2}Z_2)$ with covariance
			$$
			\E[Z(s)Z(t)] = \begin{cases}
				\E[Z_1(s)Z_1(t)] = \E[|\langle g_1,u_1\rangle \langle g_1,u_2\rangle|^q] - \E[|g|^q]^2 &: \quad s,t \in \Sm \\
				\E[Z_1(s)(-\frac{1}{2})Z_2(t)] = -\frac{1}{2}\E[(|\langle g_1,u_1\rangle|^q-\E[|g|^q])g_{t_1}g_{t_2}]   &: \quad s\in \Sm, t=(t_1,t_2)\in \{1,\ldots,m\}^2\\
				\E[(-\frac{1}{2})Z_2(s)(-\frac{1}{2})Z_2(t)] = \frac{1}{4}\Sigma_{s,t} &: \quad s,t \in \{1,\ldots,m\}^2.
			\end{cases}
			$$
			Writing $e(u) := (e_{k,\ell}(u))_{k,\ell=1}^m$ and $Z_N(u) = (Z_N(u)^{k,\ell})_{k,\ell}^m$, Slutsky's theorem provides the joint convergence of 
			$$ 
			\Big(\frac{1}{\sqrt{N}}\sum_{i=1}^N (|\langle g_i,\,\cdot\,\rangle|^q-\E[|g|^q]), \sqrt{N}(F_N-\operatorname{Id}_m), Z_N \Big) \xRightarrow[N\to\infty]{} \Big(Z_1,-\frac{1}{2}Z_2, e(\,\cdot\,)\Big).
			$$
			Finally, by the continuous mapping theorem, we obtain the weak convergence
			$$
			\frac{1}{\sqrt{N}}\sum_{i=1}^N (|\langle \sqrt{N}v_i, \,\cdot\,\rangle|^q - \E[|g|^q]) \xRightarrow[N\to\infty]{} Z(\,\cdot\,) := Z_1(\,\cdot\,) -\frac{1}{2} \sum_{k=1}^m\sum_{\ell = k}^m Z_2^{k,\ell} e_{k,\ell}(\,\cdot\,).
			$$
			Notice that the right-hand side is Gaussian (see, e.g., \cite[Page 3]{romisch2004delta}). We use Lemma \ref{lem: exact_expectations} to show that the covariance is given by 
			\begin{align*}
				& \E[Z(u)Z(v)] \\
				&=  \E[Z_1(u)Z_1(v)] -\frac{1}{2} \sum_{k=1}^m\sum_{\ell = k}^m \E[Z_1(u)Z_2^{k,\ell} e_{k,\ell}(v)]+\E[Z_1(v)Z_2^{k,\ell} e_{k,\ell}(u)]\\
				&\quad\quad\quad +\frac{1}{4} \sum_{k\ge \ell}\sum_{i\ge j} e_{k,\ell}(u)e_{i,j}(v)\E[Z_2^{k,\ell}Z_2^{i,j}]\\
				&= \E[|\langle g,u\rangle\langle g,v\rangle|^q] - \E[|g|^q]^2 \\
				&\quad\quad\quad -\frac{1}{2}\sum_{k=1}^m qu_k\E[|\langle g,u\rangle|^{q-2}\langle g,u\rangle g_k](\E[|\langle g,v\rangle|^q g_k] - \E[|g|^q])\\
				&\quad\quad\quad - \frac{1}{2}\sum_{k=1}^m qv_k\E[|\langle g,v\rangle|^{q-2}\langle g,v\rangle g_k](\E[|\langle g,u\rangle|^q g_k] - \E[|g|^q])\\
				&\quad\quad\quad-\frac{1}{2}\sum_{k=1}^m\sum_{\ell = k+1}^m (qu_k \E[|\langle g,u\rangle|^{q-2}\langle g,u\rangle g_\ell] + qu_\ell \E[|\langle g,u\rangle|^{q-2}\langle g,u\rangle g_k] )\E[|\langle g,v\rangle|^q g_k g_\ell]\\
				&\quad\quad\quad -\frac{1}{2}\sum_{k=1}^m\sum_{\ell = k+1}^m (qv_k \E[|\langle g,v\rangle|^{q-2}\langle g,v\rangle g_\ell] + qv_\ell \E[|\langle g,v\rangle|^{q-2}\langle g,v\rangle g_k] )\E[|\langle g,u\rangle|^q g_k g_\ell]\\
				&\quad\quad\quad+ \frac{1}{2}\sum_{k=1}^m q^2 u_k v_k \E[|\langle g,u\rangle|^{q-2}\langle g,u\rangle g_k]\E[|\langle g,v\rangle|^{q-2}\langle g,v\rangle g_k]\\
				&\quad\quad\quad + \sum_{k=1}^m\sum_{\ell = k+1}^m q^2 u_k v_k \E[|\langle g,u\rangle|^{q-2}\langle g,u\rangle g_\ell]\E[|\langle g,v\rangle|^{q-2}\langle g,v\rangle g_\ell]\\
				& = \E[|\langle g,u\rangle\langle g,v\rangle|^q] - \E[|g|^q]^2\\
				&\quad\quad\quad -\frac{q}{2}\sum_{k=1}^m u_k^2 \E[|g|^q]((1-qv_k^2)\E[|g|^q] - \E[|g|^q])+v_k^2 \E[|g|^q]((1-qu_k^2)\E[|g|^q] - \E[|g|^q])\\
				&\quad\quad\quad-2q^2\sum_{k=1}^m\sum_{\ell = k+1}^m u_k u_\ell v_kv_\ell \E[|g|^q]^2 +\frac{q^2}{2}\sum_{k=1}^m u_k^2 v_k^2 \E[|g|^q]^2\\
				&\quad\quad\quad + q^2 \sum_{k=1}^m\sum_{\ell = k+1}^m \E[|g|^q]^2 u_k u_\ell v_k v_\ell\\
				&= \E[|\langle g,u\rangle\langle g,v\rangle|^q] - \E[|g|^q]^2\Big(1 + \frac{q^2}{2}\sum_{k=1}^m v_k^2 u_k^2 + q^2\sum_{k=1}^m \sum_{\ell = k+1}^m u_k u_\ell v_k v_\ell \Big).
			\end{align*}
			The explicit formulas in Lemma \ref{lem: exact_expectations} and Lemma \ref{lem: asym expectation} yield the claim.
		\end{proof}
		
		\begin{proof}[Proof of Equations \eqref{eq: Hausdorff conv to ball proj} and \eqref{eq: Hausdorff conv to ball sec}]
			
			Let $V_N=(v_1,\dots,v_N)$ be uniformly distributed in $\mathbb{V}_{m,N}$ with $N\ge m$. Then an inspection of the Taylor expansion (Equation \eqref{eq: CLT proof taylor exp}) and the fact that $F_N - \operatorname{Id}_m \to 0$ a.s., gives us
			$$
			h(N^{\frac{1}{p}-\frac{1}{2}}V_N\mathbb{B}_p^N, \,\cdot\,)^q = \frac{1}{N}\sum_{i=1}^N  |\langle \sqrt{N}v_i, \,\cdot\, \rangle|^q \xrightarrow{a.s.} \E[|g|^q] = \frac{2^{q/2}\Gamma(\frac{1+q}{2})}{\sqrt{\pi}}
			$$
			in $C(\Sm, \|\,\cdot\,\|_\infty)$ (see Lemma \ref{lem: asym expectation}). Remember that a random element from the Stiefel manifold has the same distribution, up to an orthogonal transformation, as the projection onto a random subspace. Hence, for the conclusion of the projection it is sufficient to show 
			$$
			d_H\Big(N^{\frac{1}{p}-\frac{1}{2}}(V_N\mathbb{B}_p^N), \frac{\sqrt{2}}{\pi^{\frac{1}{2q}}} \Gamma(\frac{q+1}{2})^{1/q}\mathbb{B}_2^m\Big) \xrightarrow{a.s.} 0.
			$$
			The continuous mapping theorem applied to $x\mapsto x^{1/q}$ in $C(\Sm, \|\,\cdot\,\|_\infty)$ and to the function that maps the support function to its corresponding convex body give the conclusion of the theorem for the projection. For the result for the sections let $V_N\sim \operatorname{Unif}(\mathbb{V}_{m,N})$ and $E = \operatorname{Range}(V_N^*)$. Then, for each $v$ in the unit sphere of $E$, there exists a unique $u\in\Sm$ such that $V_N^*u = v$ and
			$$
			\rho(\B_p^N\cap E, v) = \rho(\B_p^N \cap E, V_N^*u) = \frac{1}{\|V_N^*u\|_p} = \frac{1}{h(V_N\B_q^N, u)},
			$$
			which we showed in Lemma \ref{lem: support function representation} and Equation \eqref{eq: 1 dim projection identity}. This gives the analogue result for sections.
		\end{proof}
		
		\begin{proof}[Proof of Theorem \ref{thm: Functional MDP}]
			We follow the proof of Theorem \ref{thm: Functional CLT} and adopt the notation. By Equation \eqref{eq: CLT proof taylor exp},
			\begin{align*}
				&\frac{1}{\beta_N\sqrt{N}}\sum_{i=1}^N (|\langle \sqrt{N}v_i, u\rangle|^q - \E[|g|^q])\\
				&\overset{d}{=} \frac{1}{\beta_N\sqrt{N}}\sum_{i=1}^N |\langle g_i,u\rangle|^q -\E[|g|^q] + \sum_{k=1}^m\sum_{\ell = k}^m (\frac{\sqrt{N}}{\beta_N}(F_N - \operatorname{Id}_m))_{k,\ell} \frac{1}{N}\sum_{i=1}^N\int_0^1  \frac{\dint}{\dint m_{k,\ell}}\Phi(\operatorname{Id}_m + t(F_N-\operatorname{Id}_m),u,g_i)\dint t\\
				&= F(X_N(u), Y_N, Z_N(u)).
			\end{align*}
			Here $F: C(\Sm)\times \overline{S^m_+} \times C(\Sm)^{m\times m}  \to C(\Sm)$, $\overline{S^m_+}\subseteq \R^{m \times m}$ denoting the space of positive semi-definite and symmetric matrices, is linear and continuous and the random variables are defined by 
			\begin{align*}
				&X_N(u) := \frac{1}{\beta_N\sqrt{N}}\sum_{i=1}^N (|\langle g_i, u\rangle|^q - \E[|g|^q]),\\
				&Y_N := \frac{1}{\beta_N\sqrt{N}}(F_N- \operatorname{Id}_m)\\
				& Z^{k,\ell}_N(u) = \frac{1}{N}\sum_{i=1}^N\int_0^1  \frac{\dint}{\dint m_{k,\ell}}\Phi(\operatorname{Id}_m + t(F_N-\operatorname{Id}_m),u,g_i)\dint t,
			\end{align*}
			where $Z^{k,\ell}_N(\,\cdot\,)$ is $(k,\ell)$-th component of $Z_N$. For each collection $\{u_1, u_2,\ldots \} = \Sm\cap \Q^m$, the sequence 
			$$
			\Big( \frac{1}{\beta_N\sqrt{N}}\sum_{i=1}^N \Big(|\langle g_i, u_1\rangle|^q - \E[|g|^q],\ldots,|\langle g_i, u_k\rangle|^q - \E[|g|^q]\Big),   \frac{1}{\beta_N\sqrt{N}} G_NG_N^* - \operatorname{Id}_m\Big)
			$$
			satisfies an MDP with speed $\beta_N^2$ in $\R^{k}\times \overline{S^m_+}$ by Proposition \ref{prop: MDP in R^m} since $q<2$. The rate function is given by $I_k:\R^{k}\times\overline{S^m_+}\to[0,\infty]$,
			$$
			I_k(x,y)= \frac{\Big\langle \begin{pmatrix} x \\ y \end{pmatrix}
				,\mathcal{C}_k^{-1}\begin{pmatrix} x \\ y \end{pmatrix}
				\Big\rangle}{2},
			$$ 
			where $\mathcal{C}_k$ is the covariance matrix of 
			$$
			\Big(\Big(|\langle g_i, u_1\rangle|^q - \E[|g|^q],\ldots,|\langle g_i, u_k\rangle|^q - \E[|g|^q]\Big),  gg^* - \operatorname{Id}_m\Big).
			$$
			Notice that no component is a multiple of any other component, so there is no non-trivial linear combination of the entries that is a constant random variable and consequently the covariance matrix is positive definite (and hence invertible). By Lemma \ref{lem: LDP for function}, the sequence $(X_N, \frac{1}{\beta_N\sqrt{N}}G_NG_N^* - \operatorname{Id}_m)$ satisfies a weak LDP with speed $\beta_N^2$ and good rate function 
			$$
			J_1(f,y) = \sup_{k\in\N}I_k(f(u_1),\ldots,f(u_k),y).
			$$
			Note that the rate function on Polish spaces is unique as soon as the weak LDP holds (\cite[Page 118]{dembo2009techniques}). The sequence $X_N$ also satisfies an LDP since it converges to zero in probability (by Slutsky's theorem and Theorem \ref{lem: CLT for gaussian sup version of ellp}) and 
			$$
			\E[e^{\lambda \sup_u||\langle g_i, u_1\rangle|^q - \E[|g|^q]|}]\le  \E[e^{\lambda \|g_i\|^q + \E[|g|^q]}] < \infty
			$$
			for all $\lambda>0$ because $q<2$ (see e.g. \cite[Page 3049]{hu2003moderate}). The sequence $\frac{1}{\beta_N\sqrt{N}}G_NG_N^* - \operatorname{Id}_m$ clearly also satisfies the MDP in $\overline{S^m_+}$ by Proposition \ref{prop: MDP in R^m}. Since both sequences are exponentially tight, the tuple is also exponentially tight (Lemma \ref{lem: exp tight produkt space}). The delta method (Proposition \ref{prop: delta method for LDPs}) applied to $(f,M) \mapsto (f,M^{-\frac{1}{2}})$ gives that 
			$$
			\frac{\sqrt{N}}{\beta_N}\Big(\big(\frac{1}{N}\sum_{i=1}^N|\langle g_i,\,\cdot\,\rangle|^q\big), F_N) - (\E[|g|^q], \operatorname{Id}_m)\Big) = \Big(\frac{1}{\beta_N\sqrt{N}}\sum_{i=1}^N (|\langle g_i,\,\cdot\,\rangle|^q-\E[|g|^q]), \sqrt{N}(F_N-\operatorname{Id}_m)\Big) 
			$$
			satisfies an MDP with speed $\beta_N^2$ and rate function $J_1((\,\cdot\,),-2(\,\cdot\,))$. Further $Z_N$ satisfies an MDP with speed $N$, because of Equation \eqref{eq: CLTproofZ_NasTupel} and since 
			$$
			\sup_u \frac{\dint}{\dint m_{i,j}}\Phi(M,u,g) = \sup_u q|\langle Mg,u \rangle |^{q-2}\langle Mg,u \rangle (g_j u_i+g_i u_j) \le c \|g\|^{q}
			$$
			for some $c\in(0,\infty)$ which satisfies the assumptions of Cramér's theorem on Banach spaces because $q<2$ (see, e.g., \cite{bahadur1979large}). Hence, the tuple $(X_N,Y_N,Z_N)$ is exponentially equivalent to $(X_N,Y_N,c_1)$ with respect to the speed $\beta_N$, where $c_1$ is the constant introduced in Equation \eqref{eq: CLT proof def c1}. The good rate function for $(X_N,Y_N,c_1)$ is $J_2(f,y,x) = J_1(f,y)$ if $x=c_1$ and $J_2(f,y,x) = \infty$ otherwise. Applying the contraction principle (Proposition \ref{prop: contraction principle}) to $F$ yields the MDP. The good rate function is given by 
			\begin{align}\label{eq: MDP rate function}
				J(f) = \inf_{F^{-1}(f) = (Z_1,Z_2, c_1)} \sup_{k\in\N} \frac{\langle \widehat{u},\mathcal{C}_k^{-1}\widehat{u}\rangle}{2},
			\end{align}
			where $\mathcal{C}_k$ is the covariance matrix of 
			$$
			\Big(\Big(|\langle g_i, u_1\rangle|^q - \E[|g|^q],\ldots,|\langle g_i, u_k\rangle|^q - \E[|g|^q]\Big),  gg^* - \operatorname{Id}_m\Big),
			$$
			and $\widehat{u} = (Z_1(u_1),\ldots,Z_1(u_k), -2Z_2)$. Here $Z_1\in C(\Sm)$ and $Z_2\in \overline{S^m_+}$. Notice that this function is positively homogeneous of degree $2$.
		\end{proof}
		
		\begin{rmk}\label{rmk: MDP rate function}
			In this remark we give the explicit form of the rate function in a comprehensive way. The rate function $J:C(\Sm) \to [0,\infty]$ is defined by 
			\begin{align*}
				J(f) = \inf_{F^{-1}(f) = (Z_1,Z_2, c_1)} \sup_{k\in\N} \frac{\langle \widehat{u},\mathcal{C}_k^{-1}\widehat{u}\rangle}{2},
			\end{align*}
			where $\mathcal{C}_k$ is the covariance matrix of 
			$$
			\Big(\Big(|\langle g, u_1\rangle|^q - \E[|g_1|^q],\ldots,|\langle g, u_k\rangle|^q - \E[|g|^q]\Big),  gg^* - \operatorname{Id}_m\Big),
			$$
			where $g = (g_1,\ldots,g_m)\sim \mathcal{N}(0,\operatorname{Id}_m)$, $g^*$ is its transposition, and $\widehat{u} = (Z_1(u_1),\ldots,Z_1(u_k), Z_2)$. Here the infimum runs over $Z_1\in C(\Sm)$ and $Z_2\in \overline{S^m_+}\subseteq \R^{m\times m}$, $\overline{S^m_+}$ being the space of symmetric and positive semi-definite matrices. Further, the constant matrix $c_1$ is given by 
			$$
			(c_1)_{k,\ell} = \E[ \frac{\dint}{\dint m_{k,\ell}} |\langle Mg_i, u\rangle|^q ],
			$$ 
			where the derivative is taken in $\overline{S^m_+}$. The linear function $F: C(\Sm)\times \overline{S^m_+} \times C(\Sm)^{m\times m}  \to C(\Sm)$ is defined by 
			$$
			F(f, (M_{1,1},M_{1,2},\ldots,M_{m,m},h_{1,1},h_{1,2},\ldots,h_{m,m}) =  f -\frac{1}{2}\sum_{k=1}^m\sum_{\ell = k}^m M_{\ell,k} h_{\ell,k} .
			$$
		\end{rmk}

		\subsection{Proof of the LDP}\label{sec: LDP proof}
		In this section, we prove the large deviation principle. We make use of the Sanov-type result \cite[Theorem 2.8]{kim2021large} for columns of random elements in the Stiefel manifold and combine this with the explicit form of the support function.\\
		
		\noindent Remember the notation from the statement of the main results. We denote by $\mathscr{M}_q(\R^m)$ the set of probability measures on $\R^m$ equipped the $q$-Wasserstein topology, which is defined as the coarsest topology that makes integration of continuous functions bounded by $x\mapsto c(1+\|x-x_0\|_2^q)$ a continuous operation, where $x_0\in \R^m$ and $c>0$ is some constant. \\
		
		\noindent The following result is also proven in \cite[Theorem 2.8]{kim2021large} and states an LDP for the empirical measure of random elements from the Stiefel manifold.
		
		\begin{lemma}\label{lem: LDP for empirical measure}
			For all $N\ge m$, let $V_N = (v_1,\ldots,v_N)$ be independently and uniformly distributed in $\mathbb{V}_{m,N}$. Then, the sequence of random measures
			$$
			\mathscr{L}_N := \frac{1}{N}\sum_{i=1}^N \delta_{\sqrt{N}v_i}
			$$
			satisfies, for $N\to\infty$, an LDP on $\mathscr{M}_q(\R^m)$ for any $0< q< 2$ with speed $N$ and strictly convex rate function $I_{\mathrm{Stiefel}}:\mathscr{M}_1(\R^m)\to [0,\infty]$, 
			$$
			I_{\mathrm{Stiefel}}(\nu) = \begin{cases}
				H(\nu|\gamma^{\otimes m}) + \frac{1}{2}\operatorname{Trace}(\operatorname{Id}_m - \mathscr{C}(\nu)) & :  \quad\operatorname{Id}_m - \mathscr{C}(\nu) \text{ positive semidefinite },\\
				+\infty & : \quad \text{ otherwise. }
			\end{cases}
			$$
			Here $\mathscr{C}$ is the covariance operator introduced in Equation \eqref{eq: covariance operaotor C}.
		\end{lemma}
		
		
		\begin{proof}[Proof of Theorem \ref{thm: Functional LDP}] We study the sequence
			$$
			X_N(\,\cdot\,) \coloneqq \frac{1}{N}\sum_{i=1}^N |\langle \sqrt{N}v_i, \,\cdot\,\rangle|^q
			$$
			of random functions on $\Sm$. Let $\{u_1,\ldots,u_k\} \subseteq \Sm$ be a subset of an enumeration of a dense subset in $\Sm$ and let $F^{(k)}:\mathscr{M}_q(\R^m)\to \R^k$, 
			$$
			F^{(k)}(\mu) = \Big(\int_{\R^m}|\langle \sqrt{N}v_i, u_1\rangle|^q \mu(\dint x), \ldots, \int_{\R^m}|\langle \sqrt{N}v_i, u_k\rangle|^q \mu(\dint x) \Big).
			$$
			This function is continuous and hence, the contraction principle (Proposition \ref{prop: contraction principle}) applies. We have
			$$
			F^{(k)}(\mathscr{L}_N) = \Big(\frac{1}{N}\sum_{i=1}^N (|\langle \sqrt{N}v_i, u_1\rangle|^q,\ldots,\frac{1}{N}\sum_{i=1}^N |\langle \sqrt{N}v_i, u_k\rangle|^q\Big),
			$$
			where we recall 
			$$
			\mathscr{L}_N = \frac{1}{N}\sum_{i=1}^N \delta_{\sqrt{N}v_i}.
			$$
			This sequence satisfies an LDP with speed $N$ by Lemma \ref{lem: LDP for empirical measure} for each $k$. Lemma \ref{lem: LDP for function} implies the weak LDP for $X_N(\,\cdot\,)$ with speed $N$ and rate function $I^\prime_{\mathrm{LDP}}: (C(\Sm), \|\,\cdot\,\|_\infty) \to [0,\infty]$, 
			$$
			I_{\mathrm{LDP}}^\prime(f) = \sup_{k\in\N} \inf\{ I_{\mathrm{Stiefel}}(\mu): F(\mu) = (f(u_1),\ldots,f(u_k))\}
			$$
			where $I_{\mathrm{Stiefel}}$ is the rate function from Lemma \ref{lem: LDP for empirical measure}. For the full LDP, notice that for $L>0$ the set of $L$-Lipschitz functions bounded by $L$, denoted by $\operatorname{Lip}_L(\Sm, [0,L])$, is compact by Arzel\`a-Ascoli's theorem (see \cite[Chapter 7, Theorem 17]{kelley2017general}). Further,
			$$
			\frac{1}{N} \sum_{i=1}^N |\langle \sqrt{N}v_i, u_1\rangle|^q \le \frac{1}{N} \sum_{i=1}^N \|\sqrt{N}v_i\|_2^q = \int_{\R^m} \|\sqrt{N}v_i\|_2^q \mathscr{L}_N(\dint x)
			$$
			and the Lipschitz constant of $X_N$ is upper bounded by 
			$$
			\int_{\R^m} q\|\sqrt{N}v_i\|_2^q \mathscr{L}_N(\dint x)
			$$
			almost surely by Lemma \ref{lem: lipschitz condition}. Hence, by \cite[Lemma 1.2.15]{dembo2009techniques},
			\begin{align*}
				&\limsup_{N\to\infty}\frac{1}{N}\log(\P(X_N \not\in \operatorname{Lip}_L(\Sm, [0,L])))\\
				&= \max\Big\{\limsup_{N\to\infty}\frac{1}{N}\log(\P(\int_{\R^m} q\|\sqrt{N}v_i\|_2^q \mathscr{L}_N(\dint x)\ge L)),\limsup_{N\to\infty}\frac{1}{N}\log(\P(\int_{\R^m} \|\sqrt{N}v_i\|_2^q \mathscr{L}_N(\dint x)\ge L)) \Big\}.
			\end{align*}
			Since $\mathscr{L}_N$ satisfies the LDP with speed $N$ and $q<2$ the same holds for both integrals in the last expression as those are continuous transformations of $\mathscr{L}_N$ in the $q$-Wasserstein topology. The LDP implies exponential tightness. This also makes $(X_N)_{N\in\N}$ exponentially tight in the uniform topology which concludes the proof.
		\end{proof}
		
		\subsection{From functional limit theorems to the volume}\label{sec: From functional limit theorems to the volume}
		Our goal is now to deduce the corresponding results for the volume of random sections and projections of $\ell_p$-balls. The radial or support function of the random projection or random section, which is itself a convex body, is a continuous (or Hadamard differentiable) transformation of the functional results (Lemma \ref{lem: support function representation} and Lemma \ref{lem: section formula}). The contraction principle (Proposition \ref{prop: contraction principle}) or the delta method (Proposition \ref{prop: delta method for LDPs} and Proposition \ref{prop: delta method infinite dim}) allow us to transfer the results to the volume. In fact, the LDP result can be deduced similarly for any continuous functional from the support or radial function. For the MDP and CLT, similar results can be obtained for any Hadamard differentiable map from the support or radial function. Also the support and radial functions for projections and sections satisfy the same weak convergence and the same MDP results (see Lemma \ref{lem: had diff for sup to rad}).\\
		
		The map from the radial function to the volume $f \mapsto \kappa_m\int_{\Sm} f^m(u) \sigma(\dint u)$ is Hadamard differentiable. This is because of Lemma \ref{lem: hadamard dir exp alpha}, since linear functions are Hadamard differentiable and the composition of Hadamard differentiable functions is again Hadamard differentiable.\\
		
		\begin{proof}[Proof of Theorem \ref{thm: LDP for volume of ell_p balls}] By Theorem \ref{thm: Functional LDP} and Lemma \ref{lem: support function representation} the normalized support function of $V_N\B_p^N$ satisfies the LDP with speed $N$. The map $\Phi$, defined in Equation \eqref{eq: def of Phi}, as well as the map 
			$$
			f \mapsto \kappa_m\int_{\Sm} f^m(u) \,\sigma(\dint u)
			$$ 
			are continuous, implying that its composition, i.e., the map to the volume of the projection, is also continuous. For the sections case, apply the continuous transformation given by Lemma \ref{lem: section formula}. The contraction principle (Proposition \ref{prop: contraction principle}) yields the claim for both cases.
		\end{proof}
		
		\begin{proof}[Proof of Theorem \ref{thm: clt for ellp balls 2}] We apply the delta method (Proposition \ref{prop: delta method infinite dim}) to the function $f\mapsto f^{\frac{m}{q}}$ using Lemma \ref{lem: hadamard dir exp alpha} and the process in the statement of Theorem \ref{thm: Functional CLT}. By the representation of the support function (Lemma \ref{lem: support function representation}), this means that 
			$$
			\Sm \ni u \mapsto \sqrt{N}\Big(h(N^{\frac{1}{2}-\frac{1}{q}} V_N \mathbb{B}^N_p, u)^m - \E[|g|^q]^{\frac{m}{q}}\Big)
			$$
			converges weakly to a centered Gaussian process $\widehat{Z}$ on the sphere with 
			\begin{align*}
				&\E[\widehat{Z}(u)\widehat{Z}(v)] = \frac{m^2}{q^2}\E[|g|^q]^{\frac{2m}{q}-2}\E[Z(u)Z(v)]\\
				&= \frac{m^2}{q^2}\E[|g|^q]^{\frac{2m}{q}-2}\Big(\E[|\langle g,u\rangle\langle g,v\rangle|^q] - \E[|g|^q]^2\Big(1 + \frac{q^2}{2}\sum_{k=1}^m v_k^2 u_k^2 + q^2\sum_{k=1}^m \sum_{\ell = k+1}^m u_k u_\ell v_k v_\ell \Big)\Big) .
			\end{align*}
			By Lemma \ref{lem: had diff for sup to rad} and Proposition \ref{prop: delta method infinite dim} the same conclusion holds if we replace the support with the radial function. Now, we apply the function $f\mapsto \kappa_m\int_\Sm f(u) \sigma(\dint u)$. The continuous mapping theorem implies that
			$$
			\sqrt{N}\Big(\vol_m(N^{\frac{1}{p}-\frac{1}{2}}(\mathbb{B}_p^N \big| E_N)) -  \kappa_m\E[|g|^q]^\frac{m}{q}\Big)
			$$
			converges in distribution to a centered Gaussian random variables with variance
			\begin{align*}
				&\V\Big[\kappa_m \int_{\Sm}\widehat{Z}(u)\sigma(\dint u)\Big] = \kappa_m^2 \E\Bigg[\Big(\int_{\Sm}\widehat{Z}(u)\sigma(\dint u)\Big)^2\Bigg]\cr
				&= \kappa_m^2\frac{m^2}{q^2}\E[|g|^q]^{\frac{2m}{q}-2} \int_{\Sm}\int_{\Sm} \E[|\langle g,u\rangle\langle g,v\rangle|^q] - \E[|g|^q]^2\Big(1 + \frac{q^2}{2}\sum_{k=1}^m v_k^2 u_k^2 \Big) \sigma(\dint u)\sigma(\dint v).
			\end{align*}
			Notice that the last part involving $\int_\Sm u_k u_\ell \sigma(\dint u)$ vanishes due to symmetry. We can calculate the integral using Lemma \ref{lem: expectation gauss part with int} and $\int_{\Sm} u_k^2 \sigma(\dint u) = \frac{1}{m}$, $1\le k \le m$, which gives us
			\begin{align*}
				&\V\Big[\kappa_m \int_{\Sm}\widehat{Z}(u)\sigma(\dint u)\Big]\\
				&= \kappa_m^2\frac{m^2}{q^2}\E[|g|^q]^{\frac{2m}{q}-2} \Big( \kappa_m^2\frac{2^{q+1}\Gamma(q+m/2)\Gamma(\frac{1+q}{2})^2\Gamma(1+\frac{m}{2})}{m\pi \Gamma(\frac{m+q}{2})^2} - \kappa_m^2\E[|g|^q]^2\big(1+\frac{q^2}{2m}\big)\Big).
			\end{align*}
			Simplifying this and using Lemma \ref{lem: asym expectation} for exact value of $\E[|g|^q]$ gives
			$$
			\frac{m \pi^{\frac{qm-m}{q}}\big(2^{q/2}\Gamma(\frac{q+1}{2})\big)^{\frac{2m}{q}}\big(4\Gamma(1+m/2)\Gamma(m/2 +q) - (2m+q^2)\Gamma(\frac{m+q}{2})^2 \big)}{2q^2\Gamma(1+m/2)^2\Gamma(\frac{m+q}{2})^2}.
			$$
			This concludes the case for projections. The section case can be done analogously by changing the roles of $p$ and $q$ and the maps $f\mapsto f^{\frac{m}{q}}$ and $f\mapsto f^{-\frac{m}{p}}$ (see Lemma \ref{lem: section formula}). This gives
			$$
			\frac{m \pi^{\frac{pm+m}{p}}\big(2^{p/2}\Gamma(\frac{p+1}{2})\big)^{-\frac{2m}{p}}\big(4\Gamma(1+m/2)\Gamma(m/2 +p) - (2m+p^2)\Gamma(\frac{m+p}{2})^2 \big)}{2p^2\Gamma(1+m/2)^2\Gamma(\frac{m+p}{2})^2}
			$$
			for the variance.
		\end{proof}
		
		\begin{proof}[Proof of Theorem \ref{thm: MDP for volume of ell_p balls}] By Lemma \ref{lem: support function representation} and Theorem \ref{thm: Functional MDP}, the sequence 
			$$
			\Bigg( \frac{\sqrt{N}}{\beta_N}\Big(h(N^{\frac{1}{2}-\frac{1}{q}} V_N \mathbb{B}^N_p, \,\cdot\,)^q - \E[|\langle \sqrt{N}v_i, \,\cdot\,\rangle|^q]\Big) \Bigg)_{N\in\N}
			$$
			satisfies an LDP with speed $\beta_N^2$ and rate function from Theorem \ref{thm: Functional MDP}. We can also replace $\E[|\langle \sqrt{N}v_i, (\,\cdot\,)\rangle|^q]$ by $\E[|g|^q]$ by Lemma \ref{lem: asym expectation} as the sequences are exponentially equivalent. By Proposition \ref{prop: delta method for LDPs} and Lemma \ref{lem: had diff for sup to rad}, we have the same result for the radial function. As in the proof of Theorem \ref{thm: clt for ellp balls 2} we apply this to the Hadamard differentiable map to the volume. The rate function transforms according to the delta method in Proposition \ref{prop: delta method for LDPs} (remember that $J_2$ was defined in Equation \eqref{eq: MDP rate function}):
			\begin{align}\label{eq: MDP rate function volume}
				I_{\mathrm{MDP}}^\prosec(x) = \inf_{\dint\Phi_\prosec(\mu_\prosec)[y] = x}J_2(y)
			\end{align}
			where 
			$$
			\Phi(f) = \begin{cases}
				\kappa_m\int_{\Sm} f^{\frac{m}{q}}(u) \sigma(\dint u) & : \quad\prosec = \big|,\\
				\kappa_m\int_{\Sm} f^{-\frac{m}{p}}(u) \sigma(\dint u) & : \quad \prosec = \cap.
			\end{cases} 
			$$
			and $\mu_{\big|} = \E[|g|^q]$, $\mu_\cap = \E[|g|^p]$. 
		\end{proof}

		\subsection*{Acknowledgment} 
		
		The authors have been supported by the DFG project \emph{Limit theorems for the volume of random projections of $\ell_p$-balls} (project number 516672205).
		
		\addcontentsline{toc}{section}{References} 
		
		\bibliographystyle{plain}
		\bibliography{bibliography}
		
		\bigskip
		\bigskip
		
		\medskip
		
		\small

		\medskip
		\noindent \textsc{Joscha Prochno:} Faculty of Computer Science and Mathematics, University of Passau, Dr.-Hans-Kapfinger-Str. 30, 94032 Passau, Germany
		
		\noindent{\it E-mail:} \texttt{joscha.prochno@uni-passau.de}
		
		\medskip
		\noindent \textsc{Christoph Thäle:}  Faculty of Mathematics, Ruhr University Bochum, Universitätsstraße 150,
		44780 Bochum, Germany
		
		\noindent{\it E-mail:} \texttt{christoph.thaele@rub.de}
		
		\medskip
		\noindent \textsc{Philipp Tuchel:}  Faculty of Mathematics, Ruhr University Bochum, Universitätsstraße 150,
		44780 Bochum, Germany
		
		\noindent{\it E-mail:} \texttt{philipp.tuchel@rub.de}

	\end{document}